\documentclass[1996/10/24]{amsart}
\usepackage{color}
\usepackage{amssymb,stmaryrd}
\usepackage{graphicx}

\newtheorem{theorem}{Theorem}[section]
\newtheorem{algorithm}{Algorithm}[section]
\newtheorem{lemma}[theorem]{Lemma}

\theoremstyle{definition}

\newtheorem{example}[theorem]{Example}

\theoremstyle{remark}
\newtheorem{remark}[theorem]{Remark}

\numberwithin{equation}{section}
\numberwithin{figure}{section}

\usepackage{color}
\definecolor{black}{rgb}{0,0,0}

\definecolor{red}{rgb}{1,0,0}

\definecolor{blue}{rgb}{0,0,1}

\newcommand{\nn}{\nonumber}

\newcommand{\Ct}{{\mathcal T}}
\newcommand{\Ce}{{\mathcal E}}

\def\bfa{\textit{\textbf{a}}}

\def\bfn{\textit{\textbf{n}}}

\def\bfu{\textit{\textbf{u}}}
\def\bfv{\textit{\textbf{v}}}
\def\bfw{\textit{\textbf{w}}}
\def\bfx{\textit{\textbf{x}}}

\def\bfI{\textit{\textbf{I}}}

\def\bfK{\textit{\textbf{K}}}

\def\bfU{\textit{\textbf{U}}}

\def\calT{\mathcal{T}}

\begin{document}

\title[Physics-preserving IMPES schemes for two-phase flow in porous media]{A new physics-preserving IMPES scheme for incompressible and immiscible two-phase flow in heterogeneous porous media}


\author{Huangxin Chen}
\address{School of Mathematical Sciences and Fujian Provincial Key Laboratory on Mathematical Modeling and High Performance Scientific Computing, Xiamen University, Fujian, 361005, China}
\email{chx@xmu.edu.cn}

\author{Shuyu Sun$^\ast$}
\address{Computational Transport Phenomena Laboratory, Division of Physical Science and Engineering, King Abdullah University of Science and Technology, Thuwal 23955-6900, Kingdom of Saudi Arabia}
\email{shuyu.sun@kaust.edu.sa}

\thanks{$^\ast$Corresponding author: Shuyu Sun (shuyu.sun@kaust.edu.sa)}



\date{}

\begin{abstract}
In this work we consider a new efficient IMplicit Pressure Explicit Saturation (IMPES) scheme for the simulation of incompressible and immiscible two-phase flow in heterogeneous porous media with capillary pressure. Compared with the conventional IMPES schemes, the new IMPES scheme is inherently physics-preserving, namely, the new algorithm is locally mass conservative for both phases and it also enjoys another appealing feature that the total velocity is continuous in the normal direction. Moreover, the new scheme is unbiased with regard to the two phases and the saturations of both phases are bounds-preserving if the time step size is smaller than a certain value. The key ideas in the new scheme include that the Darcy flows for both phases are rewritten in the formulation based on the total velocity and an auxiliary velocity referring to as the capillary potential gradient, and the total discretized conservation equation is obtained by the summation of the discretized conservation equation for each phase. The upwind strategy is applied to update the saturations explicitly, and the upwind mixed finite element methods are used to solve the pressure-velocity systems which can be decoupled. We also present some interesting examples to demonstrate the efficiency and robustness of the new algorithm.
\end{abstract}

\maketitle

\section{Introduction}
Simulation of subsurface flow has been of importance in a wide range of industrial applications such as hydrology and petroleum reservoir engineering. In this paper we focus on developing physics-preserving schemes for the simulation of incompressible and immiscible two-phase flow in heterogeneous porous media with capillary pressure. Various types of numerical methods have been developed in literature to simulate two-phase flow in porous media. The fully implicit scheme \cite{Aziz1979,Collins1992,Dawson1997,Monteagudo2007,Wu2009,Wang2013,Zidane2015,haijian-jcp,Yang2016,Yang2018} implicitly solves all the unknowns, and it thus could deserve unconditional stability and mass conservation for both of phases. The fully implicit scheme may deserve large time step size during the simulation and the difficulty of simulation lies in the resolution of nonlinear complex systems. The IMplicit-EXplicit scheme \cite{Ascher1995,Branets2009,Frank1997,Hundsdorfer2007,Koto2009,LeeEf2018} treats the linear terms implicitly and evaluates the others explicitly, and it deserves better stability than the fully explicit scheme. Operator splitting method \cite{Abreu2009,Farago2008,Lanser1999} is another efficient approach to reduce a complex time-dependent problem into some simpler problems. The IMPES scheme was originally developed by Sheldon et al. in 1959 \cite{Sheldon1959} and Stone et al. in 1961 \cite{Stone1961}. It has been widely applied to solve the coupled nonlinear system for the two-phase flow in porous media \cite{Fagin1966,Young1983,Coats1999,Coats2001,Coats2001_2,Chen2004,Chen2006}. The main idea of the IMPES scheme is to separate the computation of pressure from that of saturation, then the pressure and saturation equations are solved using implicit and explicit time approximation schemes respectively. The IMPES scheme is simple to set up and efficient to implement, and it requires less computer memory compared with the fully implicit schemes. 

For the two-phase flow in heterogeneous porous media with capillary pressure, the saturations may be discontinuous due to different capillary pressure functions, which are often resulted from the heterogeneous permeability distribution. In this case, the standard IMPES scheme \cite{Sheldon1959,Stone1961} does not reproduce the correct solutions because the standard IMPES scheme always generates spatially continuous saturation if capillarity exists. An improved IMPES scheme was proposed by Hoteit and Firoozabadi \cite{Hoteit2008_1,Hoteit2008_2} (HF-IMPES) to treat this problem. For different capillary pressures, the HF-IMPES scheme can reproduce the saturation solution with expected discontinuity. However, both the standard IMPES scheme and the HF-IMPES scheme are only mass-conservative for the wetting phase, and thus are not mass-conservative for the total fluid mixture. Moreover, both the standard IMPES scheme and the HF-IMPES scheme might produce a wetting-phase saturation which is larger than one. Until now, various kinds of improved IMPES schemes have been introduced for the simulation of two-phase flow in porous media (cf. e.g. \cite{Young1983,Thomas1983,Watts1985,Lu2000,Qin2000,Chen2004,Aarnes2006,Chen2006,Durlofsky2007,Ern2010,Kou2010,Kou2010_2,Kou2013,Kou2014,Faigle2015,Hou2016,Presho2016,Fabien2018} and the references therein), which include the sequential IMPES scheme, the iterative IMPES scheme, the adaptive implicit techniques, et al. We also note that some improvements to IMPES scheme are mainly developed for stability in presence of large time step size, for instance, potential ordering technique \cite{Kwok2007}, adaptive implicit methods \cite{Thomas1983,Collins1992} are examples of these extensions for stability purpose. We note that the adaptive implicit methods are not fully consistent at the interfaces of implicit-explicit sub-domains.

Our objective in this paper is to propose a more physics-preserving IMPES scheme for the two-phase flow in heterogeneous porous media with capillary pressure which inherently preserves the full mass conservation locally and retains the continuity of the total velocity in its normal direction. The Darcy flows for both phases are rewritten in the formulation based on the total velocity and an auxiliary velocity referring to as the capillary potential gradient. An upwind scheme is used in the spacial discretization for the conservation equation of each phase, and then the total conservation equation is also obtained by summing the discretized conservation equations of each phase. Various upwind schemes have been proposed in the literature (cf. \cite{Brenier1991,LeeEf2018} and the references therein) for multi-phase flow in porous media. Since the main focus of this work is to develop mass-conservative property for both of phases, we only use the conventional upwind scheme for the discretization of convection term. Using the unknowns for the total velocity, the auxiliary velocity and the pressures of both phases, we apply the mixed finite element method with upwind scheme to solve the pressure-velocity system. In the fully discrete system, the total linear system for velocities, phase pressures and phase saturations can be obtained when the wetting phase saturation in the former step is given. After summing the discretized conservation equation for each phase, the total discretized conservation equation can be obtained. We remark that the total linear system can be decoupled into several steps, for instance, the coupled pressure-velocity system can be decoupled into two decoupled systems which are well-posed, and then the saturations of both phases can be explicitly updated. Since the discrete mass conservation equation for each phase are both solved in the total linear system, the mass conservation property can hold true for both phases. This is the key point to ensure our new algorithm mass-conservative for each of the two phases locally (and globally).

We mention that the new IMPES scheme is unbiased with regard to the two phases. Moreover, the proposed scheme is conditionally bounds-preserving; that is, the computed saturation of each phase always sits within its physical bounds if the time step size is smaller than a certain value. Regarding the fact that it is always difficult to solve the nonlinear complex system in fully implicit scheme, we remark that our newly proposed physics-preserving IMPES scheme may also be treated as a preconditioner for the algebraic system from the fully implicit scheme. We also note that a generalization of IMPES scheme refers to the Implicit Pressure Explicit Mass (IMPEM) scheme which is considered for the solution of multi-phase multi-component flow in porous media, and it reduces to IMPES scheme when it is applied to incompressible and immiscible problems. For any IMPEM algorithm, the phases are handled symmetrically in order to preserve the mass conservation, however, it loosens the condition that the phase saturations sum to one. In our new algorithm, the property that the phase saturations sum to one and the property of mass conservation for both phases naturally hold true.

The rest of the paper is organized as follows. In Section \ref{section2}, we introduce the mathematical model for the incompressible and immiscible two-phase flow in porous media and the conventional HF-IMPES scheme. In Section \ref{section3}, we propose the new physics-preserving IMPES scheme, and a number of desired properties of the new scheme are presented and proved in Section \ref{section4}. We discuss the numerical implementation in Section \ref{section5} and verify the efficiency and robustness of the new IMPES scheme by some numerical results in Section \ref{section6}. We make conclusions and discuss future work in Section \ref{section7}.

\section{Mathematical model and HF-IMPES scheme}\label{section2}
We now introduce the basic mathematical model for incompressible and immiscible two-phase flow in porous media. Let the wetting phase and non-wetting phase be denoted by the subscripts $w$ and $n$ respectively. Based on the mass conservation law, Darcy's law, the saturations constraint and the capillary pressure, the formulations of two-phase flow with gravity in porous media $\Omega \subset \mathbb{R}^d (d=2,3)$ are given by
\begin{align}
& \phi \frac{\partial S_\alpha}{\partial t}  + \nabla \cdot \bfu_\alpha  = F_\alpha, \qquad\quad  \ \, \qquad {\rm in} \  \Omega,  \quad \alpha = w,n, \label{conservation_law_equ}\\
& \bfu_\alpha = -\frac{k_{r \alpha}}{\mu_\alpha} \bfK (\nabla p_\alpha + \rho_\alpha g \nabla z), \qquad {\rm in} \  \Omega, \quad \alpha = w,n, \label{Darcy_law_equ}\\
& S_n + S_w  = 1, \qquad\qquad\qquad \qquad\quad \  {\rm in} \   \Omega, \label{constraint_s_equ} 
\end{align}
\begin{align}
& p_c(S_w) = p_n - p_w, \qquad\qquad\qquad \ \  \ \,  {\rm in} \  \Omega. \label{pc_equ}
\end{align}
Here $\phi$ is the porosity of the medium, $\bfK$ denotes the absolute permeability tensor, $S_\alpha$, $\bfu_\alpha$, $p_\alpha$, $F_\alpha$ are the saturation, Darcy's velocity, pressure and the sink/source term of each phase $\alpha$, $p_c$ is the capillary pressure. In (\ref{Darcy_law_equ}), $\rho_\alpha$, $k_{r \alpha}$, $\mu_\alpha$ are the density, relative permeability and viscosity of phase $\alpha$, $g$ is the magnitude of the gravitational acceleration, $z$ is the depth. The phase mobility is defined by $\lambda_\alpha = \frac{k_{r \alpha}}{\mu_\alpha}$, and the total mobility is given by $\lambda_t = \lambda_w + \lambda_n$. The fractional flow functions are also defined as $f_w = \lambda_w/\lambda_t , f_n = \lambda_n / \lambda_t$. Let $\Gamma=\partial \Omega$ be composed of $\Gamma_D$ and $\Gamma_N$ such that $\Gamma  = \Gamma_D \cup \Gamma_N$ and $\Gamma_D \cap \Gamma_N = \emptyset$, where we denote by $\Gamma_D$ the Dirichlet part of boundary and $\Gamma_N$ the Neumann part of boundary. We also denote by $\Gamma = \Gamma_{\rm in} \cup \Gamma_{\rm out}$. Here $\Gamma_{\rm in} = \{ \bfx \in \Gamma : \, \bfu_t(\bfx) \cdot \bfn(\bfx) < 0   \}$ is the inflow boundary and $\Gamma_{\rm out} = \{ \bfx \in \Gamma : \, \bfu_t(\bfx) \cdot \bfn(\bfx) \geq 0   \}$ is the outflow boundary, where $\bfu_t = \bfu_n+\bfu_w$ is the total velocity and $\bfn$ is the unit outward normal to $\Gamma$. The initial and boundary conditions are imposed to the equations (\ref{conservation_law_equ}-\ref{pc_equ}) as follows:
\begin{align}
&S_\alpha = S_\alpha^0, \quad \qquad t = 0,\label{int_cond}\\
& p_\alpha = p_\alpha^B, \quad \qquad  {\rm on} \ \Gamma_D,\quad \alpha = w,n, \label{bd_cond1}\\
& \bfu_\alpha \cdot \bfn = g^N_\alpha, \quad \ \,  {\rm on} \ \Gamma_N, \quad \alpha = w,n, \label{bd_cond2}\\
& S_\alpha = S_\alpha^B, \quad\quad\quad {\rm on} \    \Gamma_{\rm in} ,\quad \alpha = w,n. \label{bd_cond3}
\end{align}
We further assume the absolute permeability tensor $\bfK$ is heterogeneous and symmetric positive and definite, the porosity $\phi$ is time-independent and uniformly bounded below and above, and there exist positive constants $\overline{\lambda}_w,\overline{\lambda}_n,\overline{\lambda}_t$ such that the mobilities satisfy $0\leq \lambda_w(S_w) \leq \overline{\lambda}_w$, $0\leq \lambda_n(S_w) \leq \overline{\lambda}_n$, $0< \lambda_t(S_w) \leq \overline{\lambda}_t$.

In the homogeneous porous media, the saturations are continuous and the standard IMPES scheme \cite{Sheldon1959,Stone1961} have been widely used to simulate the two-phase flow in porous media. The conservation law of total volume is directly obtained by the summation of the conservation laws of each phase and the constraint of the saturations of phases (\ref{constraint_s_equ}) as follows:
\[
\nabla \cdot \bfu_t = F_t,
\]
where $\bfu_t = \bfu_w + \bfu_n$ is the total Darcy velocity and $F_t = F_w+F_n$ is the total external mass flow rate. By the Darcy's law of each phase (\ref{Darcy_law_equ}) and the equation of capillary pressure (\ref{pc_equ}), one can rewrite the total Darcy velocity into
\begin{align*}
\bfu_t = -\lambda_t \bfK \nabla p_w - \lambda_n \bfK \nabla p_c - (\lambda_w \rho_w+\lambda_n \rho_n) \bfK g \nabla z.
\end{align*}
Then the conservation law of total volume can also be rewritten into
\begin{align}
- \nabla \cdot ( \lambda_t \bfK \nabla p_w) = F_t + \nabla \cdot (\lambda_n \bfK \nabla p_c) + \nabla \cdot ((\lambda_w \rho_w+\lambda_n \rho_n) \bfK g \nabla z):= {\rm RHS}_{pres}(S_w),\label{con_dec_pres}
\end{align}
which is known as the pressure equation. For given $S_w$, the pressure equation is linear with respect to the wetting phase pressure $p_w$. For the wetting phase, substituting $\bfu_w=f_w \bfu_t + \lambda_nf_w \bfK \nabla p_c-\lambda_nf_w(\rho_w-\rho_n)\bfK g \nabla z$ into the wetting phase conservation law and noting that $\nabla p_c = \frac{d p_c(S_w)}{d S_w} \nabla S_w$, we get the following equation:
\begin{align}
\phi \frac{\partial S_w}{\partial t} = F_w - \nabla \cdot(f_w \bfu_t) - \nabla \cdot \left(\lambda_n f_w \bfK\left(  \frac{d p_c}{d S_w} \nabla S_w - (\rho_w - \rho_n)g \nabla z\right) \right):= {\rm RHS}_{sat}(p_w,S_w), \label{con_dec_sa}
\end{align}
which is known as the saturation equation. Then for given $S^n_w$ from the time step $n$, the standard IMPES scheme implicitly solves $p^{n+1}_w$ by the pressure equation (\ref{con_dec_pres}) and explicitly updates the saturation $S^{n+1}_w$ of wetting phase by the saturation equation (\ref{con_dec_sa}) at the time step $n+1$ as follows:
\begin{align*}
- \nabla \cdot ( \lambda_t (S^n_w) \bfK \nabla p^{n+1}_w)  &= {\rm RHS}_{sat}(S^n_w),\\
\phi \frac{S^{n+1}_w - S^n_w}{t_{n+1}-t_n} &=  {\rm RHS}_{sat}(p^{n+1}_w,S^n_w).
\end{align*}

In the heterogeneous porous media, the capillary discontinuity may often arise from contrast in capillary pressure functions, and the saturation is discontinuous due to the continuity of the capillary pressure and the change of the capillary function across space. For different capillary pressure functions, Hoteit and Firoozabadi \cite{Hoteit2008_1,Hoteit2008_2} proposed a revised IMPES (HF-IMPES) scheme to treat contrast in capillary pressure in heterogeneous porous media. Define $\Phi_\alpha = p_\alpha + \rho_\alpha g z$, $\alpha = w,n$ and $\Phi_c =  \Phi_n - \Phi_w = p_c +(\rho_n-\rho_w)gz$. The total velocity $\bfu_t$ can also be defined as $\bfu_t = \bfu_a + \bfu_c$, where $\bfu_a = -\lambda_t \bfK \nabla \Phi_w$ and $\bfu_a = -\lambda_n \bfK \nabla \Phi_c$. The HF-IMPES scheme is briefly summarized as follows:

{\bf Step 1.} Given $S^n_w$, seek $\bfu^{n+1}_c$ such that
\[
\bfu^{n+1}_c  = -\lambda_n(S^n_w) \bfK \nabla
\Phi_c(S^n_w).
\]

{\bf Step 2.} Given $S^n_w$ and $\bfu^{n+1}_c$, seek $\bfu^{n+1}_a$ and $\Phi^{n+1}_w$ implicitly by\begin{align*}
\nabla \cdot \bfu^{n+1}_a = F_t - \nabla \cdot \bfu^{n+1}_c, \quad \bfu^{n+1}_a = - \lambda_t(S^n_w) \bfK \nabla \Phi^{n+1}_w.
\end{align*}

{\bf Step 3.} Given $S^n_w$, $\bfu^{n+1}_a$ and $\Phi^{n+1}_w$, explicitly update the wetting and non-wetting phase saturations by
\begin{align*}
\phi \frac{S^{n+1}_w - S^n_w}{t_{n+1}-t_n} = F_w - \nabla \cdot (f_w(S^n_w) \bfu^{n+1}_a), \quad S^{n+1}_n = 1 - S^{n+1}_w.
\end{align*}

We note that the HF-IMPES scheme is locally mass-conservative only for the wetting phase, but not for the non-wetting phase, just like the standard IMPES scheme. Besides, both of the standard IMPES and HF-IMPES schemes are biased with regard to the two phases and might produce a wetting-phase saturation which is larger than one. In order to overcome such disadvantages, we will introduce in the following section a new physics-preserving IMPES scheme to solve the two-phase flow problem (\ref{conservation_law_equ}-\ref{bd_cond3}) in heterogeneous porous media.

\section{A physics-preserving IMPES scheme}\label{section3}
In this section we present a new physics-preserving IMPES scheme for incompressible and immiscible two-phase flow problem (\ref{conservation_law_equ}-\ref{bd_cond3}) in heterogeneous porous media. We use the standard notations and definitions for Sobolev spaces (cf. \cite{Adams1975}) throughout the paper. For any $D \subset \overline{\Omega}$, the inner products for any scalar functions $\psi$ and $\phi$, and vector functions $\boldsymbol{\psi}$ and $\boldsymbol{\phi}$ are define by
\[
(\psi,\phi)_D = \int_D \psi \phi \, d  \bfx, \quad  (\boldsymbol{\psi} ,\boldsymbol{\phi})_D = \int_D\boldsymbol{\psi} \cdot \boldsymbol{\phi} \, d  \bfx.
\]
When $D=\Omega$, we use the notation $\|\cdot\|_{0,D}$ to denote the $L^2$-norm on $D$. We denote by $\Ct_h$ the quasi-uniform grid on $\Omega$, $\Ce_h$ the set of all faces ($d=3$) or edges ($d=2$) of $\Ct_h$, $h_K$ the diameter of any element $K\in \Ct_h$, $h=\min_{K\in \Ct_h}h_K$. Let $K, K^\prime \in \Ct_h$ and $F =\partial K \cap \partial K^\prime$ with the outward unit normal vector $\bfn_F$ exterior to $K$. We also denote by $\llbracket \psi \rrbracket = (\psi|_{K})|_F -  (\psi|_{K^\prime})|_F$ the jump of scalar function $\psi$ across interior edges/faces $F$, $\{  \psi\} = \frac{1}{2}((\psi|_{K})|_F+(\psi|_{K^\prime})|_F)$ the average of scalar function $\psi$ across interior edges/faces $F$. For edges/faces on $\partial \Omega$, $\llbracket \psi \rrbracket $ and $\{  \psi\} $ denote $\psi$. We denote by $C$ with or without subscript a positive constant depending on the shape regularity of the meshes and the coefficients data in (\ref{conservation_law_equ}-\ref{pc_equ}).

Next, we introduce the mixed finite element spaces which will be used for the spacial discrete schemes. On the
simplicial mesh $\Ct_h$, we consider the lowest-order of Raviart-Thomas finite element space
\[
RT_0(\Ct_h)  = \{ \bfv \in L^d(\Omega): \, \forall K \in \Ct_h, \bfv = \bfa + b \bfx, \bfa \in \mathbb{R}^d, b \in \mathbb{R},  \bfx \in K, \llbracket \bfv \cdot \bfn \rrbracket =0 \  on \ \forall  F \in \Ce_h \setminus \partial \Omega\}.
\]
We define $\bfU_h =RT_0(\Ct_h)$. Let $Q_h = \{ q_h \in L^2(\Omega): \, q_h|_K \in P_0(K), \ \forall K \in \Ct_h  \}$ be the piecewise constant space and $\bfU^0_h = \{ \bfv \in \bfU_h : \bfv \cdot \bfn = 0 \ {\rm on} \ \Gamma_N \}$. When the lowest-order Raviart-Thomas mixed finite element method is used for the discretization on structured grid, the mixed finite element method based on the trapezoidal rule for integration is equivalent to the cell-centered finite difference method on structured grid.

\subsection{A physics-preserving IMPES scheme (P-IMPES)}
Define $\boldsymbol{\xi}_\alpha = \lambda_t \bfw_\alpha$ with $\bfw_\alpha = -\bfK(\nabla p_\alpha + \rho_\alpha g \nabla z)$, $\alpha = n,w$. Then we easily have $\bfu_\alpha = f_\alpha \boldsymbol{\xi}_\alpha$, where $f_\alpha$ is the fractional flow function. Define $\bfu_t = \bfu_n+\bfu_w$ and $\boldsymbol{\xi}_c  = \boldsymbol{\xi}_n - \boldsymbol{\xi}_w$, we have 
\begin{align}
\bfu_w = f_w \bfu_t - f_w f_n \boldsymbol{\xi}_c , \quad \bfu_n = f_n \bfu_t + f_w f_n \boldsymbol{\xi}_c.\label{vel_wn_exp}
\end{align}
In order to clearly illustrate the derivation of the new algorithm, we rewrite the equations (\ref{conservation_law_equ})-(\ref{Darcy_law_equ}) into the following equations:
\begin{subequations}\label{new_formul_cons_darcy}
\begin{align}
&  \phi \frac{\partial S_w}{\partial t}   +   \nabla\cdot(f_w \bfu_t) =  F_w + \nabla\cdot( f_nf_w \boldsymbol{\xi}_c ) , \ \ \quad \ \qquad {\rm in} \  \Omega,\\
&   \phi \frac{\partial S_n}{\partial t}  +   \nabla\cdot(f_n \bfu_t)  =  F_n  -   \nabla\cdot( f_nf_w \boldsymbol{\xi}_c ),  \ \qquad\qquad {\rm in} \  \Omega,\\
&   (\lambda_t \bfK)^{-1} \bfu_t  =    (\lambda_t \bfK)^{-1} f_n \boldsymbol{\xi}_c  - \nabla p_w - \rho_w g \nabla z,  \  \  \qquad {\rm in} \  \Omega, \label{darcy_1} \\
&   (\lambda_t \bfK)^{-1} \bfu_t  =   - (\lambda_t \bfK)^{-1} f_w \boldsymbol{\xi}_c  - \nabla p_n - \rho_n g \nabla z,  \qquad {\rm in} \  \Omega. \label{darcy_2}
\end{align}
\end{subequations}
We note that the Darcy flows for both phases in (\ref{darcy_1}-\ref{darcy_2}) are rewritten based on the total velocity and an auxiliary velocity referring to as the capillary potential gradient. 

Next, we start to introduce a new physics-preserving IMPES scheme for the two-phase flow problem (\ref{new_formul_cons_darcy}), (\ref{constraint_s_equ}-\ref{bd_cond3}). We assume $\bfu_t,  \boldsymbol{\xi}_c \in H({\rm div},\Omega)$, $p_\alpha \in L^2(\Omega)$ and $S_\alpha \in L^2(\Omega)$, $\alpha=w,n$.

For any $\bfv \in \bfU^0_h$, $q\in Q_h$ and $S^h_w \in Q_h$,  we define a bilinear formulation $B_\alpha(\bfv, q;S^h_w)$ as follows:
\begin{align*}
B_\alpha(\bfv, q;S^h_w) = (\nabla \cdot ( f_\alpha(S^h_w) \bfv ),q)  - \sum_{K\in \calT_h} \int_{\partial K_\alpha^- \setminus \Gamma} \llbracket f_\alpha(S^h_w)\rrbracket \bfv \cdot \bfn  \, q,
\end{align*}
where $\partial K^-_\alpha = \{ e \subset \partial K : \{\bfu^h_\alpha \cdot \bfn_e\} |_e < 0  \}$ with the normal vector $n_e$ exterior to $K$. Here, $\bfu^h_\alpha$ is the discrete velocity of phase $\alpha$. Actually, this is an upwind scheme for $f_\alpha(S^h_w)$ on $\partial K$. Indeed, if $q \in Q_h$ is piecewise constant, we can compute $B_\alpha(\bfv, q;S^h_w)$ as follows:
\begin{align*}
B_\alpha(\bfv, q;S^h_w)  &= \sum_{K \in \calT_h} \int_{\partial K} f_\alpha(S^h_w) \bfv \cdot \bfn \, q - \sum_{K\in \calT_h} \int_{\partial K_\alpha^- \setminus \Gamma} \llbracket f_\alpha(S^h_w)\rrbracket \bfv \cdot \bfn  \, q  
\\
& =\sum_{K \in \calT_h} \int_{\partial K} f_\alpha(S^{\ast,h}_{w,\alpha}) \bfv \cdot \bfn \, q ,\nn
\end{align*}
where the upwind value $S_{w,\alpha}^{\ast,h}$ in the function $\lambda_\alpha(S^{\ast,h}_{w,\alpha})$ is defined as follows:
\begin{align*}
S_{\alpha}^{\ast,h}= \begin{cases}  S^{h}_\alpha|_{K_i},\quad {\rm if} \ \{\bfu^{h}_\alpha \cdot \bfn_\gamma\}_\gamma \geq 0, \\
                     S^{h}_\alpha|_{K_j},\quad {\rm if} \ \{ \bfu^{h}_\alpha \cdot \bfn_\gamma\}_\gamma < 0,
                      \end{cases}
                      \quad
S_{w,\alpha}^{\ast,h} =          \begin{cases}        S_{w}^{\ast,h}  , \quad \quad \ \ \alpha = w,\\
                                                      1  -    S_{n}^{\ast,h},\quad \alpha = n .
                                                  \end{cases}
\end{align*}
Here we denote $\gamma = \partial K_i \cap \partial K_j$ with the normal vector $\bfn_\gamma$ exterior to $K_i$. If $\gamma \subset \Gamma_{\rm in}$, then $S_{w,\alpha}^{\ast,h}|_\gamma = P_\gamma S^B_w|_\gamma$, where $P_\gamma$ is the $L^2$-projection operator into $P_0(\gamma)$. 

We also define $B_c (\bfv,q;S^h_w)$ by
\begin{align*}
B_c (\bfv,q;S^h_w) &= \left(  \nabla \cdot(f_n(S^h_w)  f_w(S^h_w)  \boldsymbol{\xi}_{c}),q  \right)  -  \sum_{K\in \calT_h} \int_{\partial K^-_w \cap K^-_n \setminus \Gamma} \llbracket f_n (S^h_w) f_w(S^h_w)\rrbracket \bfv \cdot \bfn  \, q\\
& -  \sum_{K\in \calT_h} \int_{\partial K^-_w\setminus  ( K^-_n \cup \Gamma)} \llbracket  f_w(S^h_w)\rrbracket f_n (S^{\ast,h}_{w,n})\bfv \cdot \bfn  \, q -  \sum_{K\in \calT_h} \int_{\partial K^-_n\setminus  ( K^-_w \cup \Gamma)} \llbracket  f_n(S^h_w)\rrbracket f_w (S^{\ast,h}_{w,w})\bfv \cdot \bfn  \, q.
\end{align*}
For any $q\in Q_h$, we have 
\[
B_c (\bfv,q;S^h_w) = \sum_{K\in \calT_h} \int_{\partial K} f_n(S^{\ast,h}_{w,n})f_w(S^{\ast,h}_{w,w})\bfv \cdot \bfn  \, q.
\]

Let $J = (0,T]$, we have the following continuous-in-time and discrete-in-space nonlinear system to solve the two-phase flow problem (\ref{new_formul_cons_darcy}), (\ref{constraint_s_equ}-\ref{bd_cond3}) in porous media. We denote $\sigma_w = 1$ and $\sigma_n=-1$. For any $\bfv \in \bfU^0_h$ and $q\in Q_h$, we find $\bfu^h_t (\cdot, t)\in \bfU_h$, $\boldsymbol{\xi}^h_{c}(\cdot, t) \in \bfU_h $, $p^h_\alpha(\cdot, t)  \in Q_h$, $S^h_\alpha(\cdot, t)  \in Q_h$, $\alpha = w,n$, such that
\begin{subequations}\label{dis_total_weak_form}
\begin{align}
& \left(\phi  \frac{\partial S^h_\alpha}{\partial t} , q\right) + B_\alpha(\bfu^h_t, q;S^h_w) = (F_\alpha,q) + \sigma_\alpha B_c (\boldsymbol{\xi}^h_{c},q;S^h_w) ,\qquad \qquad\qquad   \,   \, \ \alpha = w,n, \ t \in J, \label{dis_tolf_1} \\
& (( \lambda_t \bfK)^{-1} \bfu^h_t, \bfv) -(  p^h_w, \nabla \cdot \bfv )  =  (( \lambda_t \bfK)^{-1} f_n \boldsymbol{\xi}^h_{c},\bfv) - \int_{\Gamma_D} p^B_w  \bfv \cdot \bfn - (\rho_w g \nabla z,   \bfv), \qquad \  t \in J, \label{dis_tolf_2}\\
& (( \lambda_t \bfK)^{-1} \bfu^h_t, \bfv) -(  p^h_n, \nabla \cdot \bfv )  =- (( \lambda_t \bfK)^{-1} f_w \boldsymbol{\xi}^h_{c},\bfv) - \int_{\Gamma_D} p^B_n  \bfv \cdot \bfn - (\rho_n g \nabla z,   \bfv), \qquad  t \in J, \label{dis_tolf_3}\\
& (S^h_n+S^h_w,q) = (1,q),  \qquad \qquad  \  \ \ t \in J,  \\
& (p^h_n-p^h_w,q) = (p_c(S^h_w),q), \quad \quad \  \ t \in J,\\
& (S^h_w, q) = (S^0_w,q), \quad \qquad \qquad  \quad \  \, t = 0.
\end{align}
\end{subequations}

The equations (\ref{dis_tolf_1}-\ref{dis_tolf_3}) can be obtained through multiplying equations in (\ref{new_formul_cons_darcy}) by test functions, using integration by parts and applying to the space discrete approximation.
Let ${S}^{h,n}_w \in Q_h$ be given at the time step $n$, then we find $\bfu^{h,n+1}_{t} \in \bfU_h $, $\boldsymbol{\xi}^{h,n+1}_{c} \in \bfU_h $, $p^{h,n+1}_\alpha \in Q_h$, $S^{h,n+1}_\alpha  \in Q_h$, $\alpha = w,n$ at the next time step $n+1$ as follows:
\begin{subequations}\label{new_frame}
\begin{align}
& \left(\phi  \frac{  S^{h,n+1}_\alpha -S^{h,n}_\alpha }{ t_{n+1} - t_n} , q\right) + B_\alpha(\boldsymbol{u}^{h,n+1}_t, q;S^{h,n}_w) = (F_\alpha,q) + \sigma_\alpha B_c( \boldsymbol{\xi}^{h,n+1}_{c} ,q;S^{h,n}_w ), \qquad    \ \ \  \    \alpha = w,n,   \label{sch2_nonequa_1} \\
& (( \lambda_t \bfK)^{-1} \bfu^{h,n+1}_t, \bfv) -(  p^{h,n+1}_w, \nabla \cdot \bfv )  =  (( \lambda_t \bfK)^{-1} f_n(S^{h,n}_w) \boldsymbol{\xi}^{h,n+1}_{c},\bfv) - \int_{\Gamma_D} p^B_w  \bfv \cdot \bfn - (\rho_w g \nabla z,   \bfv),  \label{sch2_nonequa_2-1}  \\
& (( \lambda_t \bfK)^{-1} \bfu^{h,n+1}_t, \bfv) -(  p^{h,n+1}_n, \nabla \cdot \bfv )  =- (( \lambda_t \bfK)^{-1} f_w(S^{h,n}_w) \boldsymbol{\xi}^{h,n+1}_{c},\bfv) - \int_{\Gamma_D} p^B_n  \bfv \cdot \bfn - (\rho_n g \nabla z,   \bfv), \label{sch2_nonequa_2-2} \\
& (S^{h,n+1}_n+S^{h,n+1}_w,q) = (1,q),   \label{sch2_nonequa_3} \\
& (p^{h,n+1}_n-p^{h,n+1}_w,q) = (p_c(S^{h,n}_w),q).  \label{sch2_nonequa_4}
\end{align}
\end{subequations}

\begin{remark}
The key point of developing the P-IMPES scheme lies in solution of the above linear system (\ref{new_frame}). In the following, 
we aim to solve (\ref{new_frame}) which can be reformulated into an equivalent decoupled system. We also remark that in (\ref{new_frame}), the mass conservation equations for both phases are solved in (\ref{sch2_nonequa_1}). This is the reason why the algorithm retains the mass conservation for each of the two phases.
\end{remark}

By summing the above discrete conservation law (\ref{sch2_nonequa_1}) for each phase and noting the constraint of the saturations of phases (\ref{sch2_nonequa_3}) and the fact that
$
 \sum_{\alpha} \sigma_\alpha B_c( \boldsymbol{\xi}^{h,n+1}_{c} ,q;S^{h,n}_w )=0,
$
we now solve the following linear system to seek $\boldsymbol{u}^{h,n+1}_{t} \in \bfU_h $, $\boldsymbol{\xi}^{h,n+1}_{c} \in \bfU_h $, $p^{h,n+1}_\alpha \in Q_h, \alpha = n,w$, such that
\begin{subequations}
\begin{align}
& \sum_{\alpha} B_\alpha(\boldsymbol{u}^{h,n+1}_t, q;{S}^{h,n}_w)  = (F_t, q) ,  \label{tol_vel_equ_sch2_1}\\
& (( \lambda_t \bfK)^{-1} \bfu^{h,n+1}_t, \bfv) -(  p^{h,n+1}_w, \nabla \cdot \bfv )  =  (( \lambda_t \bfK)^{-1} f_n(S^{h,n}_w) \boldsymbol{\xi}^{h,n+1}_{c},\bfv) - \int_{\Gamma_D} p^B_w  \bfv \cdot \bfn - (\rho_w g \nabla z,   \bfv),  \label{sch2_nonequa_2-1-s}  \\
& (( \lambda_t \bfK)^{-1} \bfu^{h,n+1}_t, \bfv) -(  p^{h,n+1}_n, \nabla \cdot \bfv )  =- (( \lambda_t \bfK)^{-1} f_w(S^{h,n}_w) \boldsymbol{\xi}^{h,n+1}_{c},\bfv) - \int_{\Gamma_D} p^B_n  \bfv \cdot \bfn - (\rho_n g \nabla z,   \bfv), \label{sch2_nonequa_2-2-s} \\
& (p^{h,n+1}_n-p^{h,n+1}_w,q) = (p_c( {S}^{h,n}_w),q).  \label{tol_vel_equ_sch2_3}
\end{align}
Note that $f_n(S^{h,n}_w) + f_w(S^{h,n}_w) =1$ and $\nabla \cdot \bfv \in Q_h$,  then by (\ref{sch2_nonequa_2-1-s},\ref{sch2_nonequa_2-2-s},\ref{tol_vel_equ_sch2_3}) we further have 
\begin{align}
\left(( \lambda_t \bfK)^{-1} \boldsymbol{\xi}^{h,n+1}_c , \bfv \right)
=\left(p_c( {S}^{h,n}_w),\nabla \cdot \bfv\right)  -\int_{\Gamma_D}
(p^B_n - p^B_w)  \bfv \cdot \bfn - \left((\rho_n- \rho_w ) g \nabla
z, \bfv\right).\label{pre_vel_equ_sch2}
\end{align}
\end{subequations}

It is easy to see that the above equation (\ref{pre_vel_equ_sch2}) is well-posed for the
solution of $\boldsymbol{\xi}^{h,n+1}_c$. Then the solution of
the linear system (\ref{tol_vel_equ_sch2_1}-\ref{tol_vel_equ_sch2_3}) can be
decoupled into two steps. Firstly, we can solve (\ref{tol_vel_equ_sch2_3}) to get $p^{h,n+1}_{nw} = p^{h,n+1}_n-p^{h,n+1}_w$, and solve (\ref{pre_vel_equ_sch2}) to get $\boldsymbol{\xi}^{h,n+1}_c$. Next, $p^{h,n+1}_w$ and
$\bfu^{h,n+1}_t$ can be obtained by (\ref{tol_vel_equ_sch2_1}) and
(\ref{sch2_nonequa_2-1-s}), then $p^{h,n+1}_n$ is
directly obtained. Now we present the new physics-preserving IMPES scheme.

\smallskip

\begin{center}
\fbox{
\parbox{16cm}{
\begin{algorithm}\label{algorithm_main}
(P-IMPES scheme) Given $S^{h,n}_w$ at the time step $n$, we seek the solutions of the linear system (\ref{sch2_nonequa_1}-\ref{sch2_nonequa_4}) at the time step $n+1$ as follows:

\smallskip

\noindent {Step 1.} Seek $p^{h,n+1}_{nw} =p^{h,n+1}_n-p^{h,n+1}_w \in
Q_h$ and $\boldsymbol{\xi}^{h,n+1}_c \in \bfU_h$ by
(\ref{tol_vel_equ_sch2_3}) and (\ref{pre_vel_equ_sch2}) respectively.

\noindent {Step 2.} Seek $p^{h,n+1}_w$ and
$\boldsymbol{u}^{h,n+1}_t$ by (\ref{tol_vel_equ_sch2_1}) and (\ref{sch2_nonequa_2-1-s}). Then $p^{h,n+1}_n$ can be updated by 
\[
p^{h,n+1}_n = p^{h,n+1}_{nw} + p^{h,n+1}_w.
\]

\noindent {Step 3.}  Update the wetting phase saturation $S^{h,n+1}_w$ by (\ref{sch2_nonequa_1}) with $\alpha = w$. Then the non-wetting phase saturation $S^{h,n+1}_n$ is piecewisely updated by $S^{h,n+1}_n = 1 -S^{h,n+1}_w $. 

\end{algorithm}
}
}
\end{center}

\smallskip

We remark that the pressure-velocity system (\ref{tol_vel_equ_sch2_1}) and (\ref{sch2_nonequa_2-1-s}) in the Step 2 of Algorithm \ref{algorithm_main} is well-posed based on the theory of mixed finite element methods \cite{Babuska1980,BF1991}. Since we directly solve the total velocity by mixed finite element method in Step 2 of Algorithm \ref{algorithm_main}, the total velocity is inherently continuous in its normal direction in this algorithm.

\section{Analysis of the P-IMPES scheme}\label{section4}
In this section we focus on the analysis of a number of desired properties of the new P-IMPES scheme for the fully mass-conservative property, the unbiased property of the solution and the bounds-preserving property of saturation for both phases.

\subsection{Local mass conservation for both phases} 
For the P-IMPES scheme, the approximate saturations of both phases satisfy the discrete mass-conservative law. Actually, for any $K \in \Ct_h$, taking $q = 1$ in $\overline{K}$ and $q = 0$ in $\Omega\setminus \overline{K}$ in (\ref{sch2_nonequa_1}), we have the following local mass-conservative property for both phases.

\begin{lemma}
For any $K \in \Ct_h$, the approximate saturations of both phases based on the P-IMPES scheme satisfy the  local mass-conservative property on $K$ as follows:
\begin{align*}
\int_{K} \phi  \frac{  S^{h,n+1}_\alpha -S^{h,n}_\alpha }{ t_{n+1} - t_n}  + \int_{\partial K} f_\alpha(S^{\ast,h,n}_{w,\alpha}) \bfu^{h,n+1}_t \cdot \bfn  = \int_{K}F_\alpha + \sigma_\alpha \int_{\partial K} f_n(S^{\ast,h,n}_{w,n})f_w(S^{\ast,h,n}_{w,w}) \boldsymbol{\xi}^{h,n+1}_c  \cdot \bfn,\quad   \alpha = w,n,
\end{align*}
where $S^{\ast,h,n}_{w,\alpha}$ is the upwind value of $S^{\ast,h}_{w,\alpha}$ on each edge/face of $\partial K$ at the time step $n$.

\end{lemma}

\subsection{Unbiased property of the solution}
After we get the solutions for $\bfu^{h,n+1}_t$ and $\boldsymbol{\xi}^{h,n+1}_c$, we can update the wetting and non-wetting velocities $\bfu^{h,n+1}_w$ and $\bfu^{h,n+1}_n$ on each element by (\ref{vel_wn_exp}) as follows:
\begin{subequations}\label{update_uw_un}
\begin{align}
\bfu^{h,n+1}_w &= f_w(S^{h,n}_w) \bfu^{h,n+1}_t - f_w(S^{h,n}_w)  f_n(S^{h,n}_w)  \boldsymbol{\xi}^{h,n+1}_c ,\\  
 \bfu^{h,n+1}_n &= f_n(S^{h,n}_w)  \bfu^{h,n+1}_t + f_w(S^{h,n}_w)  f_n(S^{h,n}_w)  \boldsymbol{\xi}^{h,n+1}_c.
\end{align}
\end{subequations}
We note that if (\ref{sch2_nonequa_2-1-s}) and (\ref{pre_vel_equ_sch2}) hold, we directly have (\ref{sch2_nonequa_2-2-s}). Thus, we can also obtain $p^{h,n+1}_n$ and
$\boldsymbol{u}^{h,n+1}_t$ in Step 2 of Algorithm \ref{algorithm_main}, and then update $p^{h,n+1}_w$ by $p^{h,n+1}_w = p^{h,n+1}_n-(p^{h,n+1}_n-p^{h,n+1}_w)$. This indicates that the proposed P-IMPES scheme is unbiased in the solution of the phase pressures $p_n$ and $p_w$.

For the update of saturation in Step 3 of Algorithm \ref{algorithm_main}, we can also utilize another approach for Step 3 by updating the wetting and non-wetting phase saturations $S^{h,n+1}_w,S^{h,n+1}_n$ by (\ref{sch2_nonequa_1}) with $\alpha = w,n$ respectively. We claim that the two approaches are equivalent in the solution of saturations. Actually, on one hand, if we choose the second approach in Step 3 of the Algorithm \ref{algorithm_main}, then (\ref{sch2_nonequa_1}) holds true for $\alpha = n,w$. Due to the fact that (\ref{tol_vel_equ_sch2_1}) has been used in the computation, we have 
\[
 \left(\phi  \frac{  S^{h,n+1}_w -S^{h,n}_w }{ t_{n+1} - t_n} , q\right)  +   \left(\phi  \frac{  S^{h,n+1}_n -S^{h,n}_n }{ t_{n+1} - t_n} , q\right) =0,\quad q \in Q_h.
\] 
By the above equality and the assumption $S^{h,n}_w + S^{h,n}_n=1$, we directly have
\[
S^{h,n+1}_w + S^{h,n+1}_n=1.
\]
On the other hand, if we choose the first approach as the Step 3  shown in Algorithm \ref{algorithm_main}, by (\ref{sch2_nonequa_1}) with $\alpha=w$ and the facts $S^{h,k}_w + S^{h,k}_n=1,k=n,n+1$ and (\ref{tol_vel_equ_sch2_1}), we observe that (\ref{sch2_nonequa_1}) holds true for $\alpha=n$. Thus, these two approaches of update for the wetting and non-wetting saturations are equivalent. This also indicates that the proposed P-IMPES scheme is unbiased in the solution of the saturations $S_n$ and $S_w$. Since the first approach is simpler than another one, we use it in the update of saturations in Algorithm \ref{algorithm_main}.

\begin{remark}
The solutions of the wetting phase saturation in standard IMPES and HF-IMPES schemes are similar to the approach in Step 3 of the P-IMPES scheme. Since the mass-conservation property only holds true for one phase in standard IMPES and HF-IMPES schemes, the two approaches mentioned above for the update of saturations in standard IMPES and HF-IMPES schemes are not always equivalent. However, from the above discussion we can see that the two approaches for the update of saturations in the P-IMPES scheme are equivalent and the mass-conservative property holds true for both phases.
\end{remark}

\subsection{Bounds-preserving property of saturation for both phases}\label{subsect4_2}
In this section we will show that the approximate saturation of each phase from the P-IMPES scheme can be bounds-preserving if the time step size is smaller than a certain value. We propose two reasonable assumptions. On one hand, we assume that there exists a positive constant $\gamma_\alpha$ such that
\begin{align}
f_\alpha(S_\alpha) \leq \gamma_\alpha S_\alpha,\quad \alpha = w,n. \label{assum_1}
\end{align}
On the other hand, if $F_\alpha$ in (\ref{conservation_law_equ}) is a sink term, i.e., $F_\alpha\leq 0$, we assume that there exist two positive constants $\beta_1$ and $\beta_2$ such that
\begin{align}
\beta_1 S_\alpha \leq |F_\alpha| \leq \beta_2 S_\alpha.\label{assum_2}
\end{align}

In fact, the above two assumptions hold true reasonably. For instance, when the Brooks-Corey model \cite{BC1964} is applied for the two-phase flow in porous media, the phase mobilities are nonlinear functions and defined as 
$
\lambda_w(S_w) = \frac{1}{\mu_w}{\overline{S}_w^{\frac{2+3\theta}{\theta}}},\quad \lambda_n(S_w)= \frac{1}{\mu_n}\left( 1-\overline{S}_w \right)^2\left( 1- \overline{S}_w^{\frac{2+\theta}{\theta}} \right),
$
where $\theta$ is the parameter associated with pore size distribution, $\overline{S}_w$ is the effective saturation defined as $(S_w -S_{rw})/(1- S_{rn} - S_{rw} )$. Here, $S_{r\alpha}, \alpha = n,w,$ is residual phase saturation, namely, $S_\alpha \geq S_{r\alpha} > 0$. Let $\eta =\frac{2+3\theta}{\theta} >1$, then we can easily derive that
$
\lambda_w(S_w) \leq\frac{1}{\mu_w} (2S_w)^{\eta} \leq \frac{2^\eta}{\mu_w} S_w.
$
For the total mobility $\lambda_t > 0$, we can assume that $\lambda_t \geq \lambda_0 >0$. Thus we can choose $\gamma_w = \frac{2^\eta}{\mu_w \lambda_0} $ in (\ref{assum_1}). Similarly, we can also assume there exists a positive constant $\gamma_n$ such that $f_n(S_n) \leq \gamma_n S_n$. For the estimate (\ref{assum_2}) in the second assumption, this is also reasonable since the sink term $F_\alpha$ comes from the mass flow rate.

\begin{lemma}\label{timestep_lemma}
Let $\delta t =  t_{n+1}-t_n$. We assume the estimates (\ref{assum_1}) and (\ref{assum_2}) hold true. For the approximate wetting phase saturation in the P-IMPES scheme, we assume $S^{h,n}_\alpha \in (0,1)$ with tolerance saturation $S_{t\alpha}$, namely, $S^{h,n}_\alpha \geq S_{t\alpha}$, $\alpha=w,n$. For any $K \in \Ct_h$, if $\frac{\delta t}{h} $ is sufficiently small, we have 
\[
S^{h,n+1}_\alpha \in (0,1),\quad \alpha=w,n.
\]
\end{lemma}
\begin{proof}
We first analyze the upper bound of the wetting phase saturation. Let $\bfw^{h,n+1}_w \in \bfU_h$ be defined as $\bfw^{h,n+1}_w\cdot \bfn |_{F} = \left(\bfu^{h,n+1}_t \cdot \bfn - f_n(S^{\ast,h,n}_{w,n})  \boldsymbol{\xi}^{h,n+1}_c  \cdot \bfn  \right)|_F$ on any $F\in \Ce_h$. For any $K\in \Ct_h$, let $q = 1$ in $\overline{K}$ and $q = 0$ in $\Omega\setminus \overline{K}$ in (\ref{sch2_nonequa_1}), we have
\begin{align*}
\phi \frac{S^{h,n+1}_w-S^{h,n}_w}{\delta t}|K|  =  - \left(\int_{\widetilde{\partial K}^{-}_{w,n+1}} + \int_{\widetilde{\partial K}^{+}_{w,n+1}} \right) f_w(S^{\ast,h,n}_{w,w}) \bfw^{h,n+1}_w\cdot \bfn + \overline{F}_w|K| ,  
\end{align*}
where $\widetilde{\partial K}^{-}_{w,n+1} =  \{ F \subset \partial K : \ \bfw^{h,n+1}_w\cdot \bfn < 0  \} $ and $\widetilde{\partial K}^{+}_{w,n+1} =  \{ F \subset \partial K : \ \bfw^{h,n+1}_w\cdot \bfn \geq 0  \} $ are the inflow and outflow edges of $\bfw^{h,n+1}_w$ on $\partial K$, and $\overline{F}_w$ is the mean value of $F_w$ over $K$. By the shape regularity of the mesh, there holds
\begin{align}
\phi S^{h,n+1}_w = \phi S^{h,n}_w  -\left( \sum_{F \in \widetilde{\partial K}^{+}_{w,n+1}} + \sum_{F \in \widetilde{\partial K}^{-}_{w,n+1}} \right)\frac{\delta t}{C_Fh_K}f_w(S^{\ast,h,n}_{w,w})\bfw^{h,n+1}_w\cdot \bfn +\delta t \overline{F}_w \quad {\rm on} \  K, \label{lemm4_4_de1}
\end{align}
where $C_F>0$ depends on the shape regularity of the mesh. We have $\overline{F}_w \geq 0$ if there is injection source or no source on $K$. Otherwise, $\overline{F}_w< 0$ is production term, and by (\ref{assum_2}), we have $\overline{F}_w \geq -\beta_2 S^{h,n+1}_w$. 

By the assumption (\ref{assum_1}), we derive that
\begin{align*}
&\sum_{F \in \widetilde{\partial K}^{+}_{w,n+1}} \frac{\delta t}{C_Fh_K}f_w(S^{\ast,h,n}_{w,w})\bfw^{h,n+1}_w\cdot \bfn \\
& =\left( \sum_{F \in \widetilde{\partial K}^{+}_{w,n+1} \cap \partial K^+_{w,n}}  +  \sum_{F \in \widetilde{\partial K}^{+}_{w,n+1} \setminus \partial K^+_{w,n}}  \right)\frac{\delta t}{C_Fh_K}f_w(S^{\ast,h,n}_{w,w})\bfw^{h,n+1}_w\cdot \bfn\\
& \leq  \sum_{F \in \widetilde{\partial K}^{+}_{w,n+1} }  \frac{\delta t \gamma_w}{C_Fh_K} S^{h,n}_w \bfw^{h,n+1}_w\cdot \bfn +   \sum_{F \in \widetilde{\partial K}^{+}_{w,n+1} \setminus \partial K^+_{w,n}} \frac{\delta t \gamma_w}{C_Fh_K}\left|S^{h,n}_w|_{K'} -S^{h,n}_w \right| \bfw^{h,n+1}_w\cdot \bfn ,
\end{align*}
where $K'$ is the neighboring element of $K$ and $\partial K^+_{w,n}$ is the outflow edges of $\bfw^{h,n}_w$ on $\partial K$. For $F \in \widetilde{\partial K}^{+}_{w,n+1} \setminus \partial K^+_{w,n}$, the phase velocity changes the direction on $F$ between the two time steps, then we can assume there exists a positive constant $\epsilon$ which is small enough such that $\left| \bfw^{h,n+1}_w\cdot \bfn \right| \leq \epsilon$ on $F$. Combining the above estimates and noting the tolerance saturation assumption for $S^{h,n}_{w}$, we observe that when $\frac{\delta t}{h} $ is small enough, we have $S^{h,n+1}_w \geq  \eta S_{tw}> 0$ with $0<\eta<1$ on $K$.

Then $S^{h,n+1}_n > 0$ can also be similarly derived. Noting that $S^{h,n+1}_w + S^{h,n+1}_n=1$, then we can easily combine this saturation constraint with the low bounds of $S^{h,n+1}_\alpha $ to obtain $S^{h,n+1}_\alpha< 1, \alpha = n,w$. Now we complete the proof.
\end{proof}

\begin{remark}
By the derivation in Lemma \ref{timestep_lemma}, we can obtain the stability limit of the time step size as follows:
\[
\max\{ E_w,E_n \} \delta t / h< \phi,
\]
where $E_\alpha = \max_{K\in \Ct_h}\sum_{F \in \widetilde{\partial K}^{+}_{\alpha,n+1} }  \frac{  \gamma_\alpha}{C_F}  \bfw^{h,n+1}_\alpha\cdot \bfn  + C\epsilon, \alpha=w,n$. In the numerical experiments, we can relax this stability limit of the time step size to the  Courant-Friedrichs-Lewy (CFL) condition, and the CFL number is defined as $ \|\bfu_t\|_{L^\infty(0,T;L^\infty(\Omega))} \delta t / h$. 
\end{remark}

\section{Numerical implementation}\label{section5}
In this section, we aim to describe the P-IMPES scheme in Algorithm \ref{algorithm_main} in the matrix-vector formulation which is useful for implementation. For any shape-regular structured/unstructured mesh $\Ct_h$, we denote by $N,M$ the number of edges (2D) or faces (3D) and the number of elements in $\Ct_h$ respectively. 

Since $Q_h$ is piecewise constant space, for the basis function $q_j \in Q_h$, one can simply use $q_j|_K = 1$ and $q_j |_{\Omega\setminus  {K}}= 0$ for any $K \in \Ct_h$. For the basis function $\boldsymbol{\phi}_i \in RT_0(\Ct_h)$, we write down the detailed construction of $\boldsymbol{\phi}_i$ on the 2D unstructured mesh (cf. \cite{Bahriawati2005,BF1991}) for simplicity. One can also refer to \cite{BF1991} for the construction of basis function of $\boldsymbol{\phi}_i$ on the structured mesh. Let $F_i,i=1,2,3$ be the edges of any triangular $K \in \Ct_h$ opposite to the vertices $P_i,i=1,2,3$, and let $\bfn_i$ denote the outer unit normal of $K$ along $F_i$ and $\bfn^F_{i}$ be the unit normal vector on the edge $F_i$ with a global fixed orientation. Then the local basis function $\boldsymbol{\phi}_i$ on the element $K$ can be defined as
\[
\boldsymbol{\phi}_i = \sigma_i \frac{|F_i|}{2|K|} (\bfx - P_i),
\]
where $\sigma_i = 1$ if $\bfn^F_{i}$ points outward of $K$ and otherwise $\sigma_i = -1$, $|F_i|$ is the length of $F_i$ and $|K|$ is the area of $K$.

Based on the constructions of basis functions for $RT_0(\Ct_h)$ and $Q_h$, we now introduce several matrices and vectors which will be used in the matrix-vector formulation for the P-IMPES scheme. We define
\begin{align*}
\boldsymbol{\rm A}_h = \left( \int_\Omega (\lambda_t \bfK)^{-1} \boldsymbol{\phi}_i \cdot \boldsymbol{\phi}_j \right)_{i,j = 1,\cdots,N} \in \mathbb{R}^{N\times N},  \quad
\boldsymbol{\rm B}_h  = \left( \int_\Omega q_j  \nabla \cdot  \boldsymbol{\phi}_i  \right)_{i=1,\cdots,N, j = 1,\cdots,M}  \in \mathbb{R}^{N\times M} .
\end{align*}
For any $\bfv \in RT_0(\Ct_h)$, noting that the upwind values $f_n (S^{\ast,h}_{w,n})$ and $f_w (S^{\ast,h}_{w,w})$ are both single value on each edge, we have $f_\alpha (S^{\ast,h}_{w,\alpha}) \bfv \in  RT_0(\Ct_h), \alpha = n,w$, and $f_n (S^{\ast,h}_{w,n}) f_w (S^{\ast,h}_{w,w}) \bfv \in  RT_0(\Ct_h)$. We further define
\begin{align*}
\boldsymbol{\rm B}^h_\alpha  &= \left( \int_\Omega q_j \nabla \cdot ( f_\alpha (S^{\ast,h}_{w,\alpha}) \boldsymbol{\phi}_i )  \right)_{i= 1,\cdots,N,j=1,\cdots,M} \in \mathbb{R}^{N\times M} ,
 \\
\boldsymbol{\rm B}^h_c  &= \left( \int_\Omega q_j  \nabla \cdot ( f_n (S^{\ast,h}_{w,n})f_w (S^{\ast,h}_{w,w}) \boldsymbol{\phi}_i  )  \right)_{i=1,\cdots,N, j = 1,\cdots,M} \in \mathbb{R}^{N\times M} .
\end{align*}
Let the vectors $\boldsymbol{\rm b}_D^h , \boldsymbol{\rm g}_h , \boldsymbol{\rm b}^h_c, \boldsymbol{\rm b}^h_{\alpha,D}, \boldsymbol{\rm g}^h_\alpha \in \mathbb{R}^N$ and $\boldsymbol{\rm F}^h_t, \boldsymbol{\rm F}^h_\alpha \in \mathbb{R}^M$, $\alpha = n,w$, be defined as follows:
\begin{align*}
\boldsymbol{\rm b}_D^h = \left( \int_{\Gamma_D} (p^B_n - p^B_w)  \boldsymbol{\phi}_i  \cdot \bfn  \right)_{i=1,\cdots,N}, \quad \boldsymbol{\rm g}_h = \left(\int_\Omega   (\rho_n - \rho_w) g \nabla z \cdot  \boldsymbol{\phi}_i\right)_{i=1,\cdots,N},
\end{align*}
\begin{align*}
&\boldsymbol{\rm b}_c^h(\boldsymbol{\xi},\boldsymbol{\rm S}^{h}_{w}) = \left( \int_{\Omega} (\lambda_t \bfK)^{-1} f_n (\boldsymbol{\rm S}^{h}_{w}) \boldsymbol{\xi}\cdot\boldsymbol{\phi}_i  \right)_{i=1,\cdots,N}, \quad \boldsymbol{\xi} \in RT_0(\Ct_h),\boldsymbol{\rm S}^{h}_{w} \in \mathbb{R}^M,
\\
&\boldsymbol{\rm b}^h_{\alpha,D} = \left( \int_{\Gamma_D} p^B_\alpha  \boldsymbol{\phi}_i  \cdot \bfn  \right)_{i=1,\cdots,N},\quad 
\boldsymbol{\rm g}^h_\alpha = \left(\int_\Omega   \rho_\alpha g \nabla z \cdot  \boldsymbol{\phi}_i \right)_{i=1,\cdots,N},\\
& \boldsymbol{\rm F}^h_t = \left(\int_\Omega F_t \, q_i  \right)_{i=1,\cdots,M}, \quad
\boldsymbol{\rm F}^h_\alpha = \left(\int_\Omega F_\alpha \, q_i  \right)_{i=1,\cdots,M}.
\end{align*}

Now we can start to introduce the matrix-vector formulation for the P-IMPES scheme.

\smallskip

\begin{center}
\fbox{
\parbox{16cm}{
\begin{algorithm}\label{algorithm_main0}
(P-IMPES scheme in matrix-vector formulation) Given $\boldsymbol{\rm S}^{h,n}_w \in \mathbb{R}^M$ at the time step $n$, we seek the solutions of the linear system (\ref{sch2_nonequa_1}-\ref{sch2_nonequa_4}) at the time step $n+1$ as follows:

\smallskip

\noindent {Step 1.} Let $\boldsymbol{\rm p}^{h,n+1}_{nw} = p_c(\boldsymbol{\rm S}^{h,n}_w) \in \mathbb{R}^M$ and seek $\boldsymbol{\rm x}^{h,n+1}_c \in \mathbb{R}^N$ by
\[
\boldsymbol{\rm A}_h \boldsymbol{\rm x}^{h,n+1}_c = \boldsymbol{\rm B}_h p_c(\boldsymbol{\rm S}^{h,n}_w) - \boldsymbol{\rm b}^h_{D}  - \boldsymbol{\rm g}_h.
\]

\noindent {Step 2.} Now we have $\boldsymbol{\xi}^{h,n+1}_c = \sum_{i=1}^N ( \boldsymbol{\rm x}^{h,n+1}_c)_i \boldsymbol{\phi}_i$. We further seek $\boldsymbol{\rm u}^{h,n+1}_t \in \mathbb{R}^N$ and $\boldsymbol{\rm p}^{h,n+1}_w\in \mathbb{R}^M$ by
\begin{align*}
(\boldsymbol{\rm B}^h_n + \boldsymbol{\rm B}^h_w)^T  \boldsymbol{\rm u}^{h,n+1}_t & = \boldsymbol{\rm F}^h_t,\\
\boldsymbol{\rm A}_h \boldsymbol{\rm u}^{h,n+1}_t -  \boldsymbol{\rm B}_h\boldsymbol{\rm p}^{h,n+1}_w & = \boldsymbol{\rm b}^h_c(\boldsymbol{\xi}^{h,n+1}_c,\boldsymbol{\rm S}^{h,n}_w) - \boldsymbol{\rm b}^h_{w,D} - \boldsymbol{\rm g}^h_w.
\end{align*}
Then we seek the non-wetting phase pressure $\boldsymbol{\rm p}^{h,n+1}_n\in \mathbb{R}^M$ by $\boldsymbol{\rm p}^{h,n+1}_n = \boldsymbol{\rm p}^{h,n+1}_{nw}+\boldsymbol{\rm p}^{h,n+1}_w$.

\smallskip

\noindent {Step 3.} Update the wetting phase saturation $\boldsymbol{\rm S}^{h,n+1}_w\in \mathbb{R}^M$ and $\boldsymbol{\rm S}^{h,n+1}_n\in \mathbb{R}^M$ by
\[
\frac{\phi}{\delta t}(\boldsymbol{\rm S}^{h,n+1}_w - \boldsymbol{\rm S}^{h,n}_w) + \boldsymbol{\rm B}^h_w \boldsymbol{\rm u}^{h,n+1}_t  = \boldsymbol{\rm F}^h_w + \boldsymbol{\rm B}^h_c \boldsymbol{\rm x}^{h,n+1}_c,\quad \boldsymbol{\rm S}^{h,n+1}_n = \boldsymbol{1} - \boldsymbol{\rm S}^{h,n+1}_w,
\]
where $\boldsymbol{1} \in \mathbb{R}^M$ denotes the vector with 1 as each element.

\end{algorithm}
}
}
\end{center}

\smallskip

Here, when $\boldsymbol{\rm S}^{h,n}_w$ is given, the matrices $\boldsymbol{\rm B}^h_n,\boldsymbol{\rm B}^h_w,\boldsymbol{\rm B}^h_c$ in the above matrix-vector formulation of the new P-IMPES scheme are generated based on $S^{h,n}_w = \sum_{i=1}^M (\boldsymbol{\rm S}^{h,n}_w)_i q_i$. We also remark that when $ \boldsymbol{\rm u}^{h,n+1}_t $ and $\boldsymbol{\rm x}^{h,n+1}_c$ are obtained, the wetting and non-wetting phase velocities can be obtained by (\ref{update_uw_un}) and used for the determination of upwind values of $S^h_{w,n}$ and $S^h_{w,w}$ on each edge/face at the next time step.

\section{Numerical experiments}\label{section6}
In this section, we present several numerical examples to verify the performance of our new P-IMPES scheme. We assume that the absolute permeability tensor is chosen as $\bfK = K\bfI$, where $\bfI$ is the identity matrix and $K$ is a positive real number with unit $1 \, md$. In the following experiments, the capillary pressure function  is given by Hoteit and Firoozabadi in \cite{Hoteit2008_1} as
$
p_c(S_w) = -\frac{B_c}{\sqrt{K}} {\rm log \,}\overline{S}_w,
$
where $B_c$ is a positive parameter with unit $1\, bar \cdot md^{1/2}$ and $\overline{S}_w$ is the effective saturation defined as in Section \ref{subsect4_2}. We assume the porosity of the porous media $\phi = 0.2$ in all the following examples. Moreover, the relative permeabilities are given by $k_{rw} = \overline{S}_w^\beta$ and $k_{rn} = (1-\overline{S}_w)^\beta$, where $\beta$ is a positive integer number. In the experiments, we set the residual saturations as $S_{rw} = S_{rn}=10^{-6}$ and use the quadratic relative permeabilities, i.e., $\beta = 2$.

\begin{example}\label{example0}
In this example we simulate a drainage process of wetting phase in a heterogeneous porous media with permeability distribution shown in right graph of Figure \ref{ex0_fig1}. The domain size is set as $300\, m \times 150 \,m$, and the problem is simulated on a triangular mesh with 5000 elements and mesh size $h=3\, {\rm m}$. We assume there are subdomains with permeability $K=1\, md$ and $K = 50 \, md$ respectively. The density of wetting phase is set as $\rho_{W}=1000 \, kg/m^{3}$ and the non-wetting phase is $\rho_{N}=800 \, kg/m^{3}$. {The viscosity of the wetting phase is set as $\mu_{W}=1\, cP$ and the viscosity of the non-wetting phase is $\mu_{N}=0.3\, cP$.} The injection of wetting phase is from the left boundary with a rate of $0.63 \, m^3/day$. We assume the two-phase flow is produced on the right boundary with constant pressure of wetting phase $100 \, bar$, and the rest of the boundary is impermeable. We assume the porous media is almost fully filled with non-wetting phase flow at the initial state and the elements intersected with inflow boundary are filled with wetting phase flow. We test the parameter $B_c=0$ and $B_c=60$ in the capillary pressure function and use the time step size as $0.1\, day$.

\begin{figure}[htbp]
\includegraphics[width=8cm,height=4cm,clip,trim=0cm 1cm 1cm 0cm]{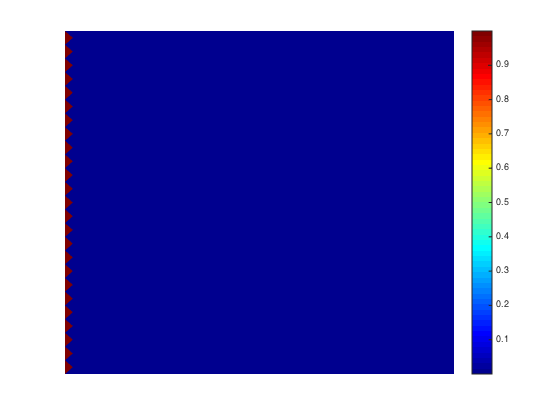}  
\includegraphics[width=8cm,height=4cm,clip,trim=1cm 1cm 0cm 0cm]{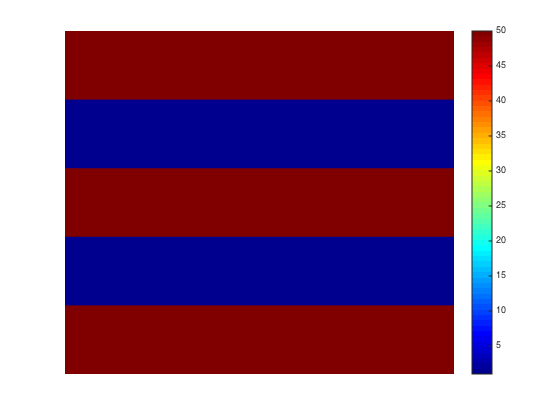}
\caption{\footnotesize (Example \ref{example0})  Left: wetting phase saturation at the initial state. Right: permeability.}\label{ex0_fig1}
\end{figure}

\begin{figure}[htbp]
\includegraphics[width=8cm,height=4cm,clip,trim=0cm 1cm 1cm 0cm]{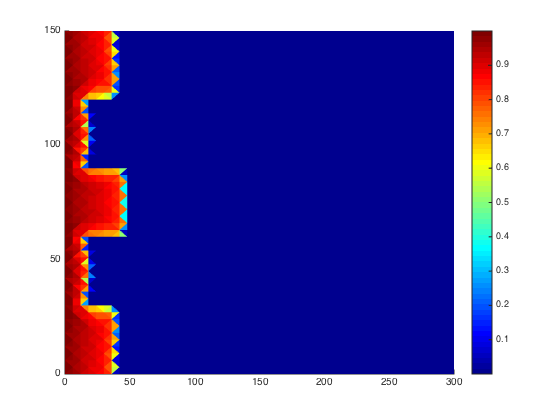}  
\includegraphics[width=8cm,height=4cm,clip,trim=1cm 1cm 0cm 0cm]{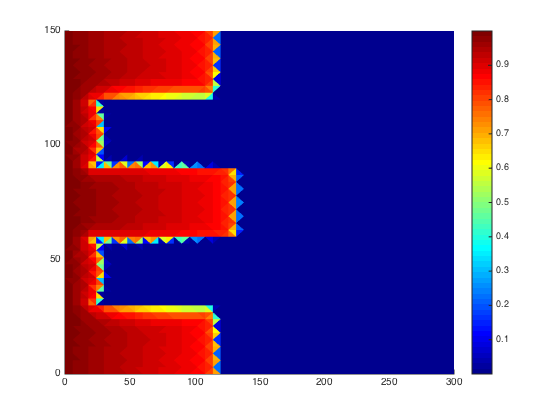}
\includegraphics[width=8cm,height=4cm,clip,trim=0cm 1cm 1cm 0cm]{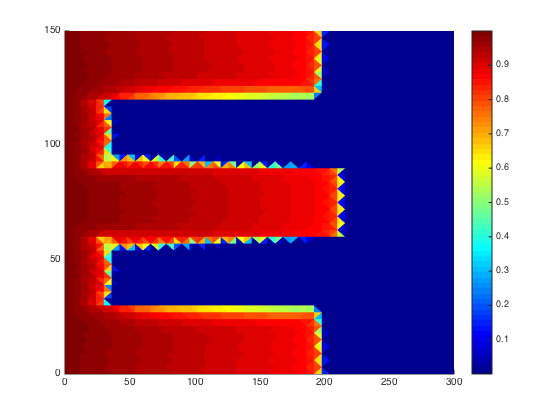}  
\includegraphics[width=8cm,height=4cm,clip,trim=1cm 1cm 0cm 0cm]{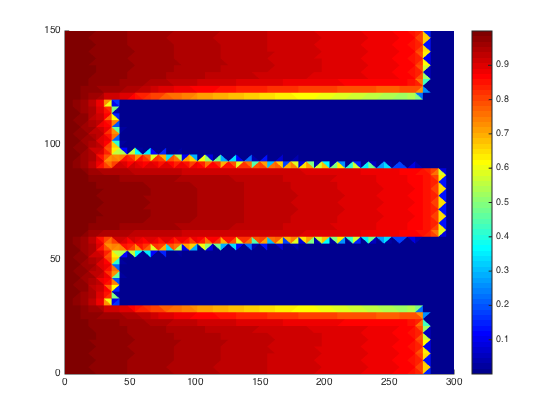}
\caption{\footnotesize (Example \ref{example0} for the case of $B_c=0$) Saturation of wetting phase in the drainage process at the time steps 100, 300, 500 and 700.} \label{ex0_fig2}
\end{figure}

\begin{figure}[htbp]
\includegraphics[width=5.5cm,height=4.5cm,clip,trim=0cm 0.5cm 0cm 0cm]{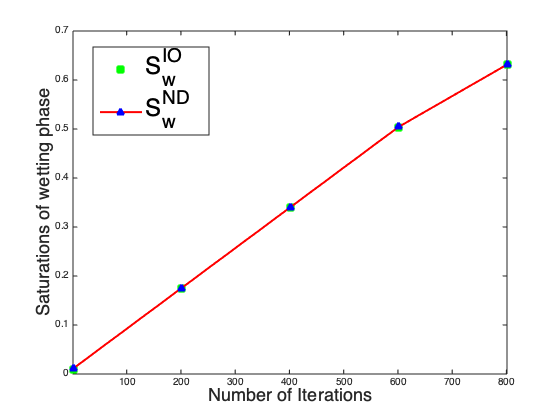}  
\includegraphics[width=5.5cm,height=4.5cm,clip,trim=0cm 0.5cm 0cm 0cm]{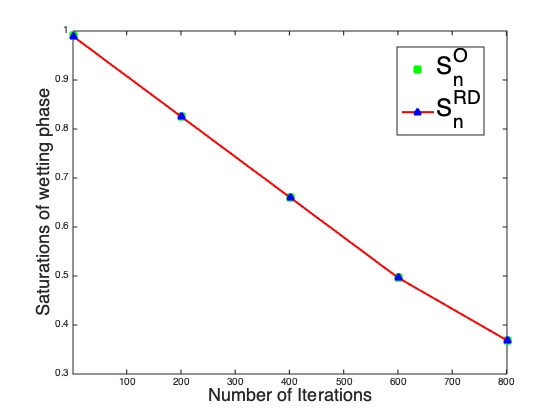}
\caption{\footnotesize (Example \ref{example0} for the case of $B_c=0$)  Left: spatial average of wetting phase saturation. Right: spatial average of non-wetting phase saturation.}\label{ex0_fig4}
\end{figure}

We firstly test the problem in a heterogeneous porous media without capillary pressure, i.e., the case of $B_c=0$. The drainage process at different times steps is illustrated in Figure \ref{ex0_fig2}. Then we further test the problem in a heterogeneous porous media with capillary pressure, i.e., the case of $B_c=60$, and the saturations of wetting phase at different time steps is shown in Figure \ref{ex0_fig3}. We can see that the drainage process for the cases of $B_c=0$ and $B_c=60$ is different. When the capillary pressure is not considered, we can see that the injected water flows faster in the more permeable layers. When we take the capillary pressure into account, the flow in the more permeable layers slows down due to the contrast in capillary pressure.

\begin{figure}[htbp]
\includegraphics[width=8cm,height=4cm,clip,trim=0cm 1cm 1cm 0cm]{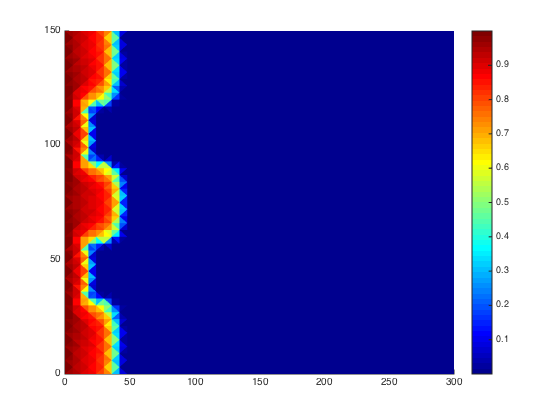}  
\includegraphics[width=8cm,height=4cm,clip,trim=1cm 1cm 0cm 0cm]{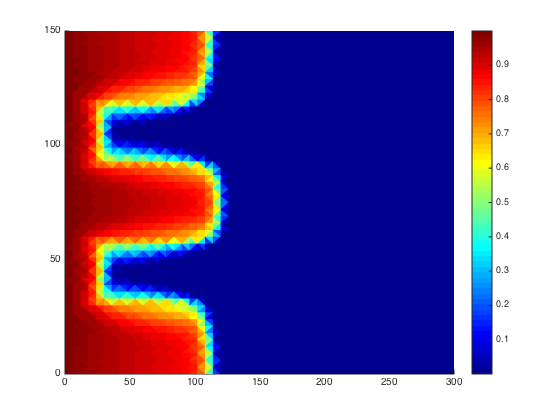}
\includegraphics[width=8cm,height=4cm,clip,trim=0cm 1cm 1cm 0cm]{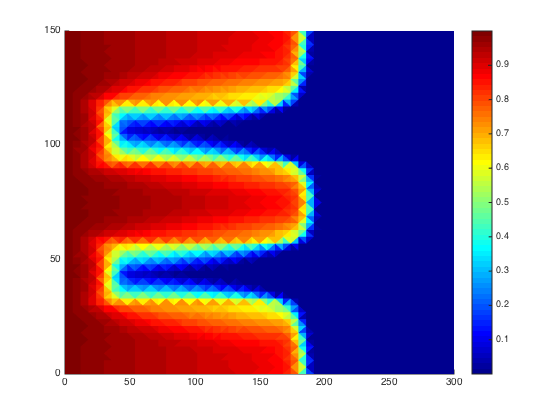}  
\includegraphics[width=8cm,height=4cm,clip,trim=1cm 1cm 0cm 0cm]{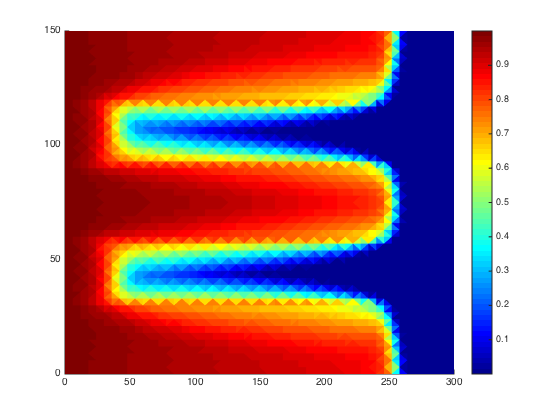}
\caption{\footnotesize (Example \ref{example0} for the case of $B_c=60$) Saturation of wetting phase in the drainage process at the time steps 100, 300, 500 and 700.} \label{ex0_fig3}
\end{figure}

\begin{figure}[htbp]
\includegraphics[width=5.5cm,height=4.5cm,clip,trim=0cm 0.5cm 0cm 0cm]{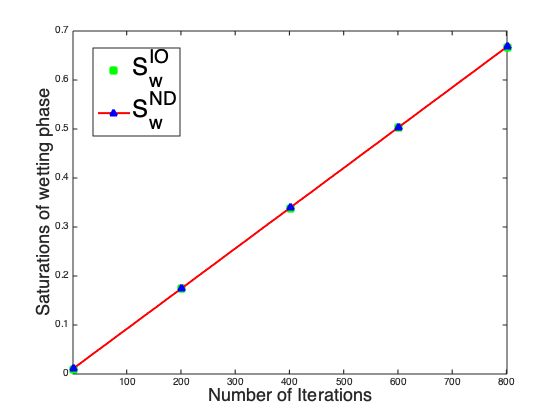}  
\includegraphics[width=5.5cm,height=4.5cm,clip,trim=0cm 0.5cm 0cm 0cm]{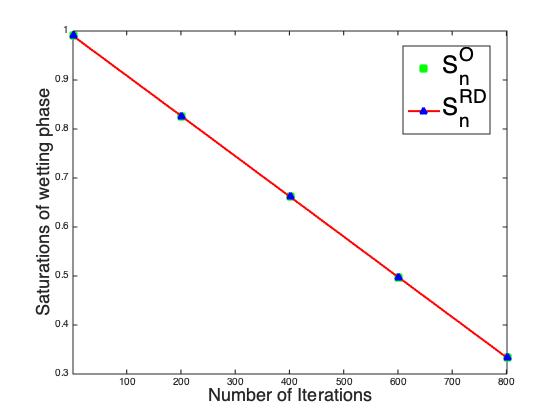}
\caption{\footnotesize (Example \ref{example0} for the case of $B_c=60$)  Left: spatial average of wetting phase saturation. Right: spatial average of non-wetting phase saturation.}\label{ex0_fig5}
\end{figure}

We remark that the mean saturation of wetting phase (spatial average of $S_w$) are calculated by the injection and the simulation results based on the P-IMPES scheme, and the mean saturation of non-wetting phase are also calculated based on the P-IMPES scheme. We denote by $S^{IO}_w$ the mean saturation based on the summation of the injected wetting phase at the new time step and the wetting phase in the porous media at the old time step, $S^{ND}_w$ the mean saturation based on the summation of the wetting phase in the porous media and the discharge of wetting phase at the new time step. Similarly, we denote by $S^{O}_n$ the mean saturation of the non-wetting phase at the old time step and $S^{RD}_n$ the mean saturation based on the summation of the residual non-wetting phase in the porous media and the discharge of non-wetting phase at the new time step. As shown in Figure \ref{ex0_fig4} and Figure \ref{ex0_fig5}, we can see that the value of $S^{IO}_w$ meets well with $S^{ND}_w$, and the value of $S^{O}_n$ also meets well with $S^{RD}_n$ for both cases of $B_c=0$ and $B_c=60$. This clearly verifies the mass conservation property of our algorithm.

\begin{figure}[htbp]
\includegraphics[width=5.5cm,height=4.5cm,clip,trim=0cm 0.5cm 0cm 0cm]{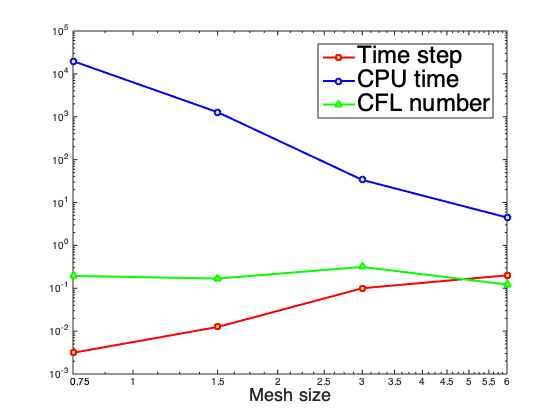}
\includegraphics[width=5.5cm,height=4.5cm,clip,trim=0cm 0.5cm 0cm 0cm]{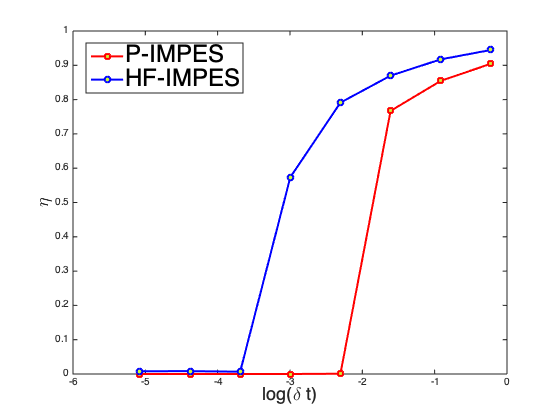}
\caption{\footnotesize (Example \ref{example0} for the case of $B_c=60$) Left: CPU time and feasible time step size for the solution of the problem on the meshes with different mesh sizes. Right: plot of the parameter $\eta$ by the P-IMPES and HF-IMPES schemes respectively.}\label{ex0_fig6}
\end{figure}

We also remark that the CFL numbers we use for the cases of $B_c=0$ and $B_c=60$ based on the P-IMPES scheme are 0.2 and 0.3 respectively on the mesh with mesh size $h=3\, {\rm m}$. The feasible time step size, the CPU time and the CFL number are shown in the left graph of Figure \ref{ex0_fig6} for the case of $B_c=60$ on the meshes with different mesh sizes for the simulations after 10 days, and we find that the complexity of the algorithm correspondingly increases on the fine mesh.

We further compare the feasible time step size which can be used for the P-IMPES and HF-IMPES schemes. Let  
\[
\eta = 1 - \frac{\|\|S^h_w(\bfx,t)\chi_{S^h_w}\|_{L^2(\Omega)}\|_{L^2(0,T)}}{\|\|S^h_w(\bfx,t)\|_{L^2(\Omega)}\|_{L^2(0,T)}},
\]
where the indicator function $\chi_{S^h_w}$ is defined as $\chi_{S^h_w} =1$ if $0\leq S^h_w\leq 1$ and otherwise $\chi_{S^h_w} =0$. We show the parameter $\eta$ in the right graph of Figure \ref{ex0_fig6} for the case of $B_c=60$ on the mesh with mesh size $h=3\, {\rm m}$ for the simulation after 10 days by the P-IMPES and HF-IMPES schemes respectively. From the theoretical analysis for the P-IMPES scheme, we know that the saturations of both phases can be bounded if the time step size is small enough, i.e., $\eta=0$ if $\delta t$ is small enough. This property may not always be true for the HF-IMPES scheme. We can see from the right graph of Figure \ref{ex0_fig6} that the feasible time step size for the P-IMPES scheme can be chosen much larger than that used for the HF-IMPES scheme.

\end{example}

\begin{example}\label{example1}
Next we simulate a drainage process of wetting phase in a heterogeneous porous media with domain size $180\, m \times 180 \,m$. We test the problem on a triangular mesh with 7200 elements.
The distribution of permeability in logarithmic value is described in the right graph of Figure \ref{ex1_fig1}. The permeability data are obtained from the SPE10 data set which can be downloaded from the SPE website (http://www.spe.org/web/csp/). We use the $60 \times 60$ cutting data of the horizontal permeability in one of the horizontal levels. The density of wetting phase is set as $\rho_{W}=1000 \, kg/m^{3}$ and the non-wetting phase is $\rho_{N}=660 \, kg/m^{3}$. {The viscosity of the wetting phase is set as $\mu_{W}=1\, cP$ and the viscosity of the non-wetting phase is $\mu_{N}=0.45\, cP$.} The injection of wetting phase is from the bottom-left boundary of one mesh size, with a rate of $1.97 \, m^3/day$, as shown in the left graph of Figure \ref{ex1_fig1} for the initial state of $S_w$. We assume the two-phase flow is produced on the top-right boundary of one mesh size with constant pressure of wetting phase $100 \, bar$, and the rest of the boundary is impermeable. We also assume the porous media is almost fully filled with non-wetting phase flow at the initial state, and use the parameter $B_c=1$ in the capillary pressure and the time step size as $0.2\, day$ in this case.

\begin{figure}[htbp]
\includegraphics[width=5.5cm,height=4cm,clip,trim=0cm 1cm 0cm 0cm]{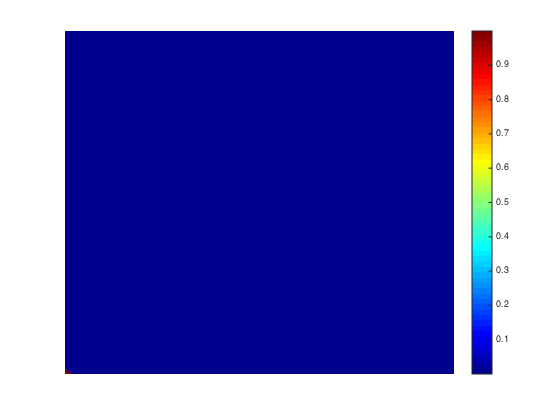}  
\includegraphics[width=5.5cm,height=4cm,clip,trim=0cm 1cm 0cm 0cm]{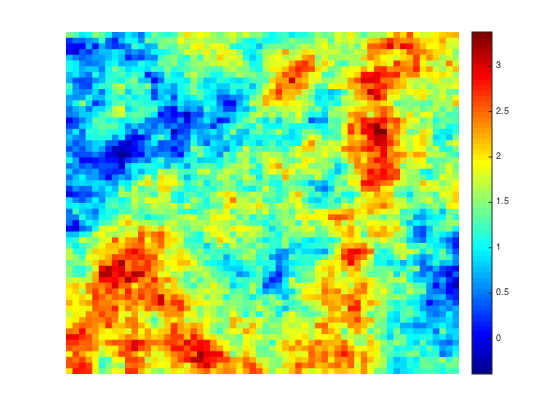}
\caption{\footnotesize (Example \ref{example1})  Left: wetting phase saturation at the initial state. Right: permeability.}\label{ex1_fig1}
\end{figure}

\begin{figure}[htbp]
\includegraphics[width=5.5cm,height=4cm,clip,trim=0cm 1cm 0cm 0cm]{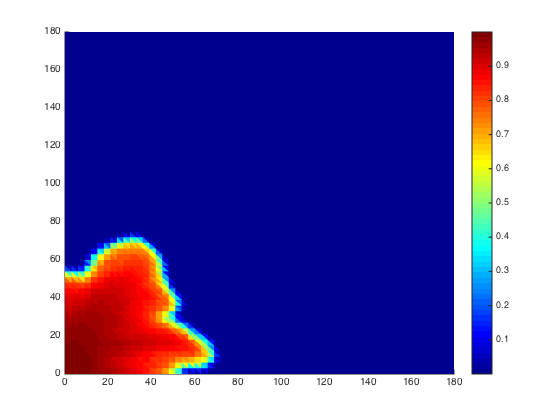}  
\includegraphics[width=5.5cm,height=4cm,clip,trim=0cm 1cm 0cm 0cm]{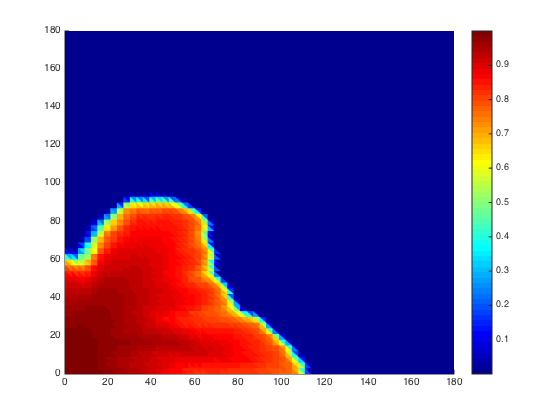}
\includegraphics[width=5.5cm,height=4cm,clip,trim=0cm 1cm 0cm 0cm]{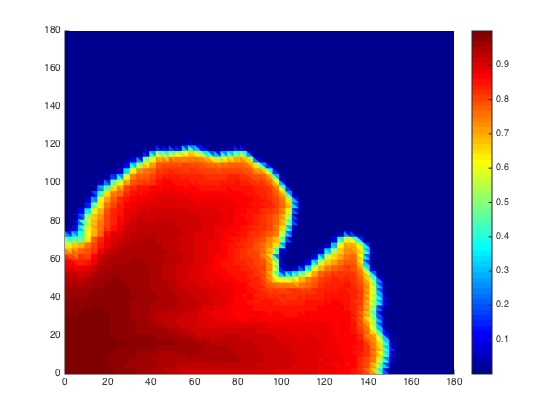}  
\includegraphics[width=5.5cm,height=4cm,clip,trim=0cm 1cm 0cm 0cm]{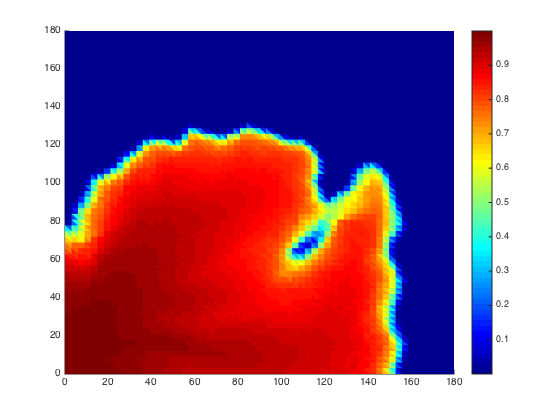}
\caption{\footnotesize (Example \ref{example1}) Saturation of wetting phase in the drainage process at the time steps 500, 1000, 2000 and 2500.} \label{ex1_fig2}
\end{figure}

\begin{figure}[htbp]
\includegraphics[width=5.5cm,height=4.5cm,clip,trim=0cm 0.5cm 0cm 0cm]{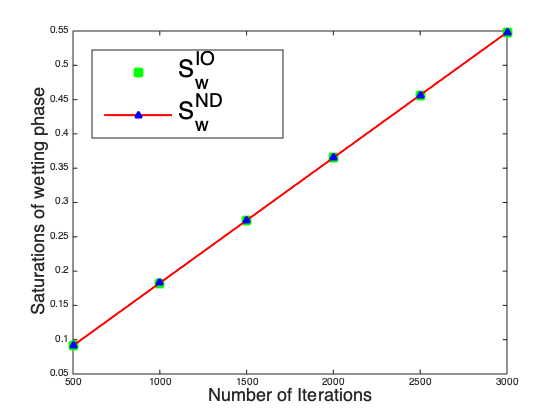}  
\includegraphics[width=5.5cm,height=4.5cm,clip,trim=0cm 0.5cm 0cm 0cm]{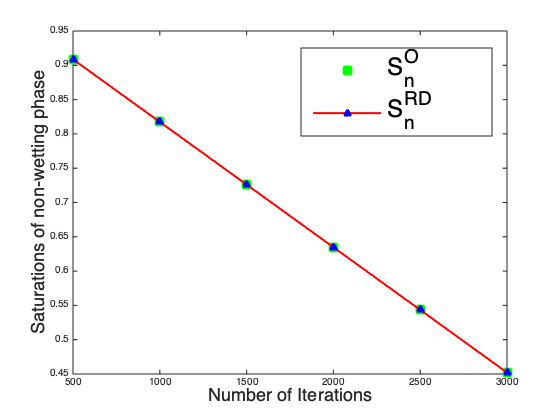}
\caption{\footnotesize (Example \ref{example1})  Left: spatial average of wetting phase saturation. Right: spatial average of non-wetting phase saturation.}\label{ex1_fig3}
\end{figure}

The drainage process at different times steps is illustrated in Figure \ref{ex1_fig2} which shows the saturations of wetting phase at different time steps. As shown in the two graphs of Figure \ref{ex1_fig3}, we can see that the value of $S^{IO}_w$ meets well with $S^{ND}_w$, and the value of $S^{O}_n$ also meets well with $S^{RD}_n$. This clearly verify the mass conservation property of our algorithm.

\end{example}

\begin{example}\label{example2}
In this example we test our algorithm for a two-phase counter flow problem with the domain size $250 \, m \times 250\, m$. As shown in the left graph of Figure \ref{ex2_fig1}, the water initially lies in the middle of lower layer, which is lighter than the heavy oil that fills out the rest of the domain. The densities of water and heavy oil are $\rho_{W}=1000 \, kg/m^{3}$ and $\rho_{N}=1200\, kg/m^{3}$ respectively, and the viscosities of water (wetting phase) and heavy oil (non-wetting phase) are set as $\mu_{W}=1\, cP$ and $\mu_{N}=0.45 \, cP$. The gravity is taken into consideration in this case. We assume the permeability $K=1 \, md$ in the left middle of the domain and $K=1000\, md$ in the rest of the domain which is shown in the right graph of Figure \ref{ex2_fig1}. We test the problem on a triangular mesh with 5000 elements
and assume the boundary is impermeable. We test the parameter in the capillary pressure function for $B_c=10$, and the time step size is set as $0.05\, day$.

\begin{figure}[htbp]
\includegraphics[width=5.5cm,height=4cm,clip,trim=0cm 1cm 0cm 0cm]{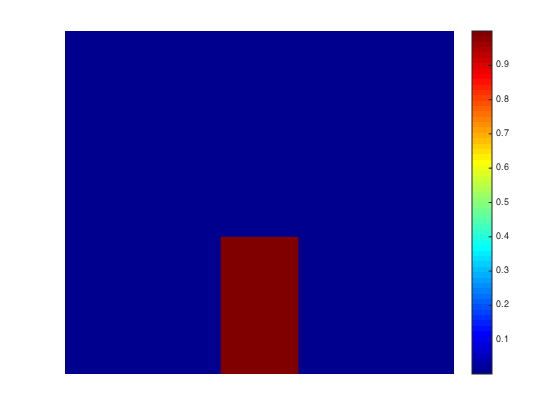}  
\includegraphics[width=5.5cm,height=4cm,clip,trim=0cm 1cm 0cm 0cm]{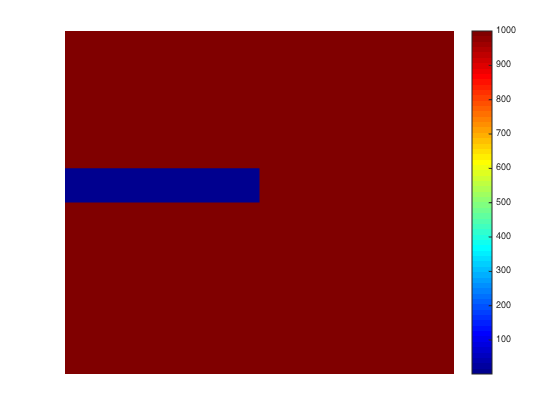}
\caption{\footnotesize (Example \ref{example2})  Left: wetting phase saturation at the initial state. Right: permeability.}\label{ex2_fig1}
\end{figure}


\begin{figure}[htbp]
\includegraphics[width=5.5cm,height=4cm,clip,trim=0cm 1cm 0cm 0cm]{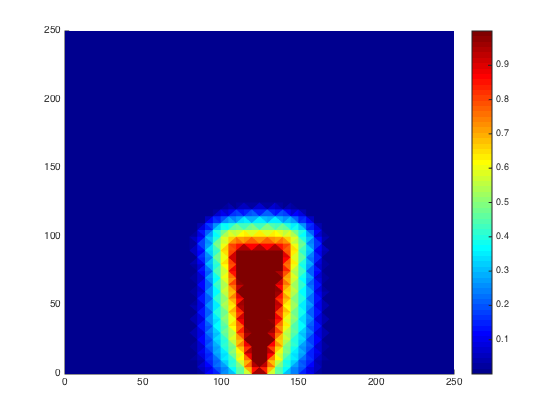}  
\includegraphics[width=5.5cm,height=4cm,clip,trim=0cm 1cm 0cm 0cm]{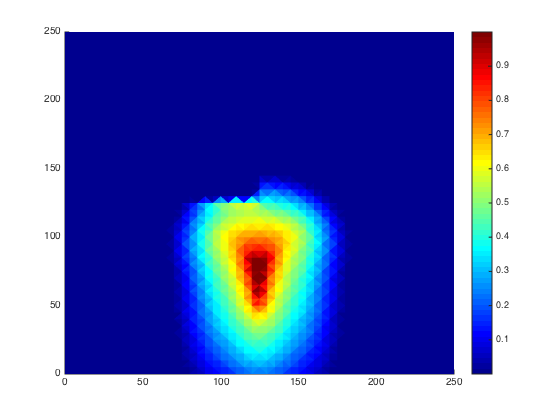}
\includegraphics[width=5.5cm,height=4cm,clip,trim=0cm 1cm 0cm 0cm]{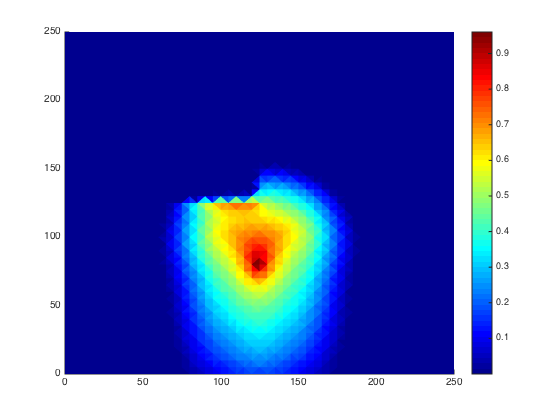}  
\includegraphics[width=5.5cm,height=4cm,clip,trim=0cm 1cm 0cm 0cm]{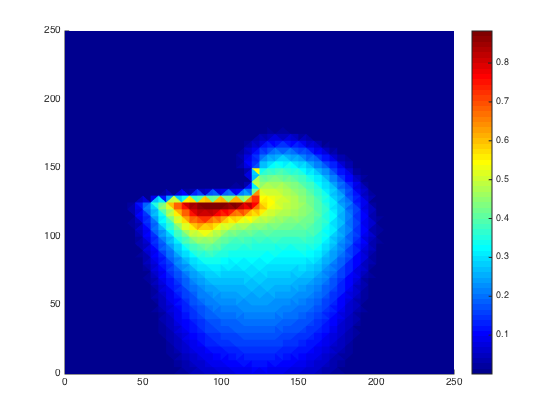}
\caption{\footnotesize (Example \ref{example2} for the case of $B_c=10$) Saturation of wetting phase in the drainage process at the time steps 300, 900, 1500 and 2400.} \label{ex2_fig3}
\end{figure}

\begin{figure}[htbp]
\includegraphics[width=5.2cm,height=4.5cm,clip,trim=0cm 0.5cm 0cm 0cm]{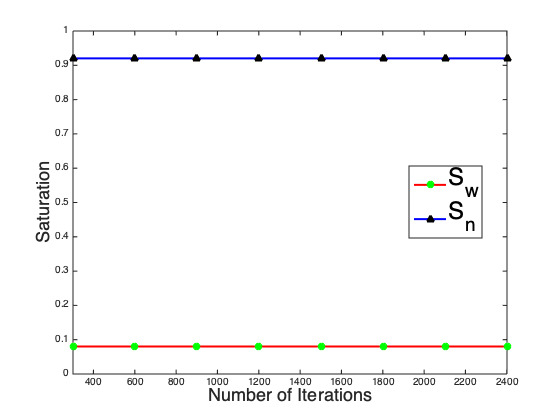}
\caption{\footnotesize (Example \ref{example2})  Spatial average of phase saturations based on our algorithm.}\label{ex2_fig4}
\end{figure}

\begin{figure}[htbp]
\includegraphics[width=5.5cm,height=4.5cm,clip,trim=0cm 0.5cm 0cm 0cm]{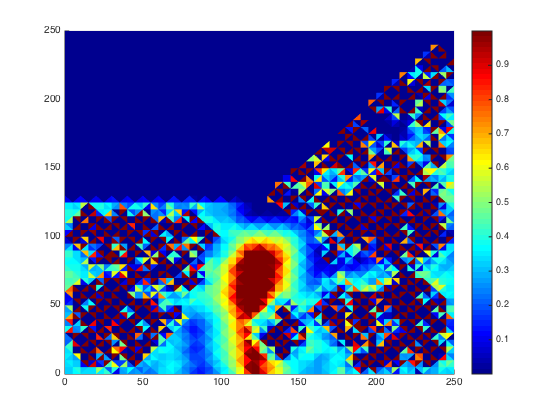}
\includegraphics[width=5.2cm,height=4.5cm,clip,trim=0cm 0.5cm 0cm 0cm]{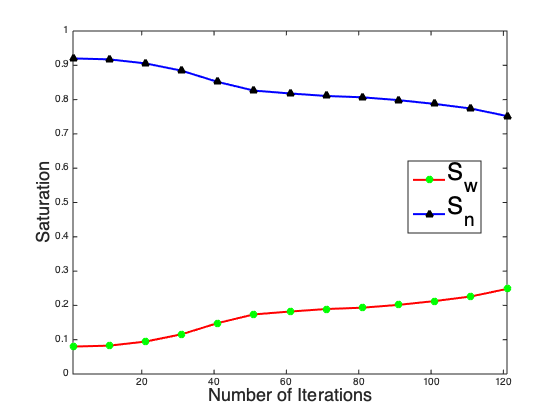}
\caption{\footnotesize (Example \ref{example2})  Left: spatial average of phase saturations based on the HF-IMPES scheme at the time step 120. Right: spatial average of phase saturations based on the HF-IMPES scheme.}\label{ex2_fig5}
\end{figure}

We show the saturations of wetting phase at different time steps in Figure \ref{ex2_fig3}. We can see that the water (the wetting phase) raises up gradually under the effect of gravity. The mean saturations of wetting and non-wetting phases are shown in Figure \ref{ex2_fig4}. Noting that we use the impermeable boundary condition, we can see that the mass conservation property holds well based on our algorithm. We also test the conventional HF-IMPES scheme for the case of $B_c=10$ with time step size $0.05\, day$. Since the saturations can not be bounded by one if the HF-IMPES scheme is used, we use the cutting method in the simulation if the saturations are larger than one. We can see from the left graph in Figure \ref{ex2_fig5} that the solution of saturation of wetting phase based on the HF-IMPES scheme blows out at the time step 120, and from the right graph in Figure \ref{ex2_fig5} we can also see that the mass-conservation property does not hold well for the conventional HF-IMPES scheme.
\end{example}

\begin{example}\label{example3}
We further test our algorithm for a two-phase counter flow problem in a heterogeneous porous media with domain size $250 \, m \times 250\, m$. As shown in the left graph of Figure \ref{ex3_fig1}, the water initially lies in part of lower layer, which is lighter than the heavy oil that fills out the rest of the domain. The densities of water and heavy oil are $\rho_{W}=1000 \, kg/m^{3}$ and $\rho_{N}=1200\, kg/m^{3}$ respectively, and the viscosities of water (wetting phase) and heavy oil (non-wetting phase) are set as $\mu_{W}=1\, cP$ and $\mu_{N}=0.45 \, cP$. The gravity is taken into consideration in this case. The permeability in logarithmic value is shown in the right graph of Figure \ref{ex3_fig1}. We test the problem on a triangular mesh with 5000 elements
and assume the boundary is impermeable. We test the parameter in the capillary pressure function for $B_c=1$ and the time step size is set as $1\, day$ in this case.

\begin{figure}[htbp]
\includegraphics[width=5.5cm,height=4cm,clip,trim=0cm 1cm 0cm 0cm]{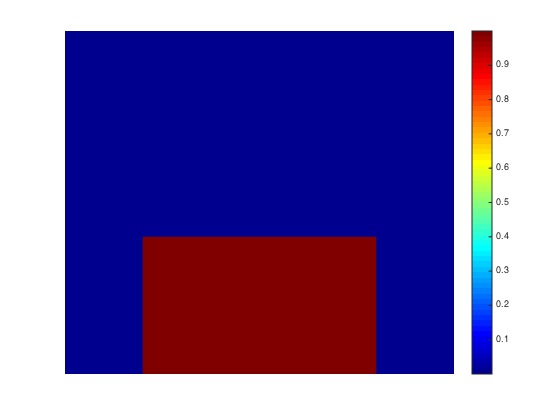}  
\includegraphics[width=5.5cm,height=4cm,clip,trim=0cm 1cm 0cm 0cm]{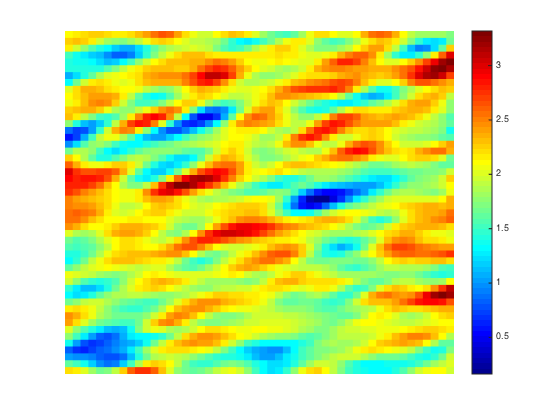}
\caption{\footnotesize (Example \ref{example3})  Left: wetting phase saturation at the initial state. Right: permeability.}\label{ex3_fig1}
\end{figure}


\begin{figure}[htbp]
\includegraphics[width=5.5cm,height=4cm,clip,trim=0cm 1cm 0cm 0cm]{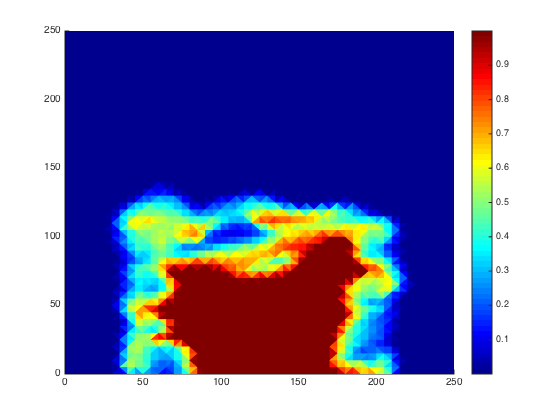}  
\includegraphics[width=5.5cm,height=4cm,clip,trim=0cm 1cm 0cm 0cm]{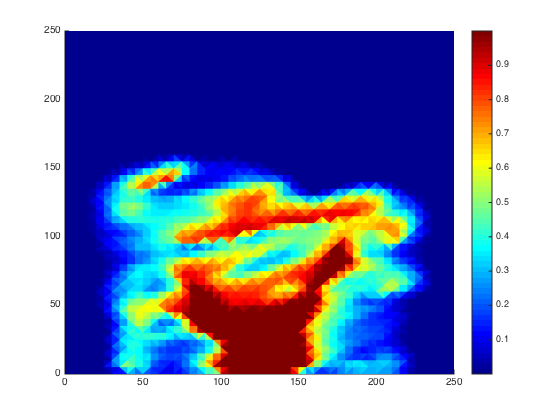}
\includegraphics[width=5.5cm,height=4cm,clip,trim=0cm 1cm 0cm 0cm]{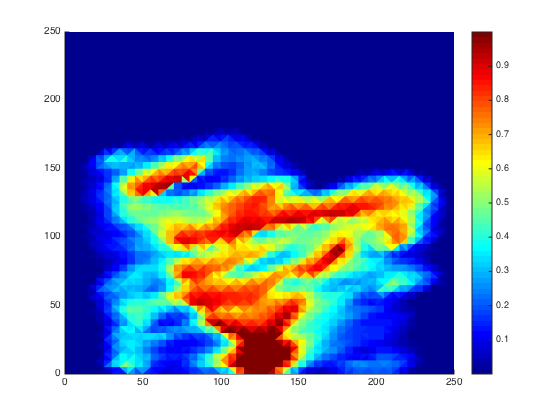}  
\includegraphics[width=5.5cm,height=4cm,clip,trim=0cm 1cm 0cm 0cm]{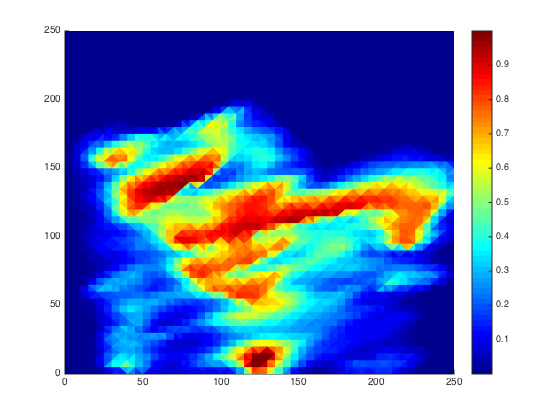}
\caption{\footnotesize (Example \ref{example3} for the case of $B_c=1$) Saturation of wetting phase in the drainage process at the time steps 300, 900, 1500 and 2400.} \label{ex3_fig3}
\end{figure}

\begin{figure}[htbp]
\includegraphics[width=5.2cm,height=4.5cm,clip,trim=0cm 0.5cm 0cm 0cm]{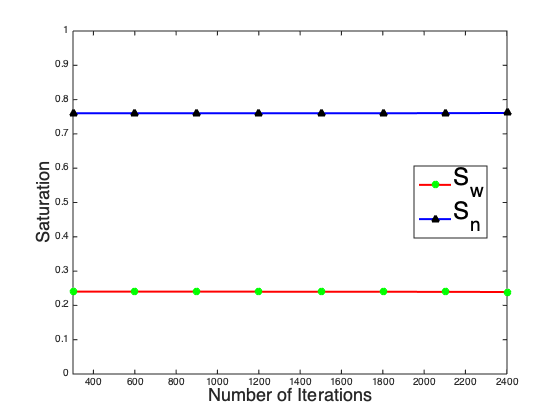}
\caption{\footnotesize (Example \ref{example3})  Spatial average of phase saturations based on the our algorithm.}\label{ex3_fig4}
\end{figure}

\begin{figure}[htbp]
\includegraphics[width=5.5cm,height=4.5cm,clip,trim=0cm 0.5cm 0cm 0cm]{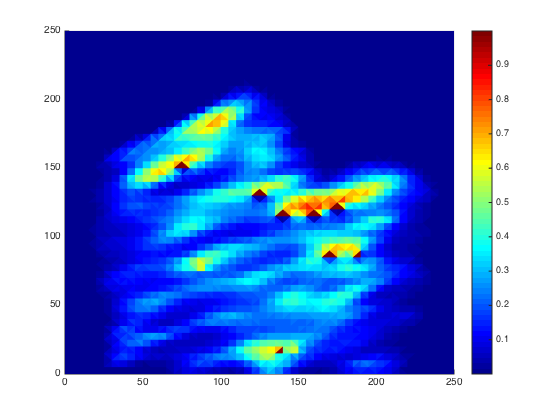}
\includegraphics[width=5.2cm,height=4.5cm,clip,trim=0cm 0.5cm 0cm 0cm]{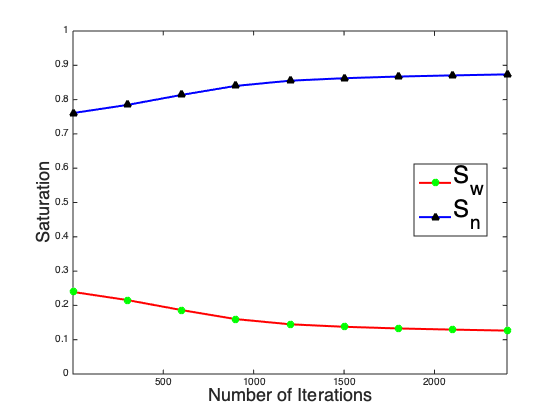}
\caption{\footnotesize (Example \ref{example3})  Left: spatial average of phase saturations based on the HF-IMPES scheme at the time step 2400. Right: spatial average of phase saturations based on the HF-IMPES scheme.}\label{ex3_fig5}
\end{figure}

We show the saturations of wetting phase at different time steps for the case of $B_c=1$ in Figure \ref{ex3_fig3}. We can also see that the water (the wetting phase) raises up gradually under the effect of gravity. The mean saturations of wetting and non-wetting phases are shown in Figure \ref{ex3_fig4}, we can see that the mass conservation property holds well based on our algorithm. We also test the conventional HF-IMPES scheme for the case of $B_c=1$. The saturations of wetting phase at the time step 2400 is shown in the left graph of Figure \ref{ex3_fig5}. We can observe that some values of wetting phase saturation are cut by one due to the fact that these values are larger than one based on the HF-IMPES scheme. From the right graph in Figure \ref{ex3_fig4} we can see that the mass-conservation property does not hold well for this case based on the conventional HF-IMPES scheme.

\end{example}

\section{Conclusion}\label{section7}
In this paper we aim to design a physics-preserving IMPES scheme for the simulation of incompressible and immiscible two-phase flow in heterogeneous porous media with capillary pressure. The key ideas of the new algorithm are to rewrite the Darcy flows for both phases in the formulation based on the total velocity and an auxiliary velocity referring to as the capillary potential gradient, and to obtain the total conservation equation by summing the discretized conservation equation for each phase. We find that the new P-IMPES scheme is locally mass conservative for both phases and it also retains the desired property that the total velocity is continuous in its normal direction. Another merit of the new algorithm is that the new scheme is unbiased with regard to the two phases and the saturation of each phase can be proved to be bounds-preserving if the time step size is small enough. For the numerical experiments that we have performed, we demonstrated that the proposed scheme always holds the mass-conservation property. We believe that our algorithm can also be extended to the compressible and immiscible two-phase ow in porous media, which is our on-going work.

\providecommand{\href}[2]{#2}

\end{document}